\documentclass[11pt]{article}

\usepackage[a4paper,left=2cm,right=2cm,top=1.5cm,bottom=2cm]{geometry}
\usepackage[utf8]{inputenc}
\usepackage[T1]{fontenc}
\usepackage[english]{babel}
\usepackage{amsthm}
\usepackage{amsmath}
\usepackage{amssymb}
\usepackage{mathrsfs}
\usepackage{setspace}
\usepackage{float}
\usepackage{multirow}
\usepackage{stmaryrd}
\usepackage{color}
\usepackage{wrapfig}
\usepackage{pifont}
\usepackage{array}
\usepackage[colorlinks=true,linkcolor=ocre,citecolor=ocre]{hyperref}
\usepackage{dsfont}
\usepackage{upgreek}
\usepackage{nicefrac}
\usepackage{tikz}{}
\usepackage{gensymb}
\usepackage{FiraSans}
\usepackage{graphicx}
\usepackage{caption}
\renewenvironment{proof}{\noindent{\sffamily{\textbf{Proof :}}}}{\begin{flushright}$\square$\end{flushright}}

\newcommand{\IE}{\mathbb{E}}

\newcommand{\N}{\mathbb{N}}

\newcommand{\IR}{\mathbb{R}}
\newcommand{\R}{\mathbb{R}}

\newcommand{\IT}{\mathbb{T}}

\newcommand{\drm}{\mathrm d}

\newcommand{\CD}{\mathcal D}
\newcommand{\CS}{\mathcal S}

\newcommand{\CE}{\mathcal E}

\newcommand{\CC}{\mathcal C}
\newcommand{\CH}{\mathcal H}

\newcommand{\CB}{\mathcal B}

\newcommand{\CW}{\mathcal W}

\newcommand{\SU}{\mathscr{U}}

\newcommand{\SF}{\mathscr{F}}

\renewcommand{\P}{\mathsf{P}}

\newcommand{\PI}{\mathsf{\Pi}}

\newcommand{\DC}{\mathsf{C}}

\newcommand{\Wick}[1]{\mathbf{:}#1\mathbf{:}}

\newcommand{\eps}{\varepsilon}

\newcommand{\dd}{\mathrm d}

\newcommand\supp{\mathrm{supp}}

\definecolor{ocre}{RGB}{64,123,121}

\newcounter{item}
\numberwithin{item}{section}

\newtheorem{theorem}[item]{\sffamily Theorem}
\newtheorem{definition}[item]{\sffamily Definition}
\newtheorem{proposition}[item]{\sffamily Proposition}
\newtheorem{lemma}[item]{\sffamily Lemma}
\newtheorem{corollary}[item]{\sffamily Corollary}
\newtheorem{remark}[item]{\sffamily Remark}

\newtheorem*{theorem*}{\sffamily Theorem}
\newtheorem*{definition*}{\sffamily Definition}
\newtheorem*{proposition*}{\sffamily Proposition}
\newtheorem*{lemma*}{\sffamily Lemma}
\newtheorem*{corollary*}{\sffamily Corollary}
\newtheorem*{remark*}{\sffamily Remark}

\usepackage[explicit]{titlesec}
\titleformat{\section}{\centering\Large\bfseries}{\thesection \ --}{0.7em}{\Large\bfseries #1}
\titleformat{\subsection}{\centering\large\bfseries}{\thesubsection \ --}{0.4em}{\large\bfseries #1}
\titleformat{\subsubsection}{\centering\bfseries}{\thesubsubsection \ --}{0.4em}{\bfseries #1}

\let\emph\relax
\DeclareTextFontCommand{\emph}{\bfseries\em}

\usepackage{mathtools}
\numberwithin{equation}{section}
\mathtoolsset{showonlyrefs}

\title{\bfseries Nonlinear Schrödinger equations with spatial white noise potential on full space for $d\le 3$}
\author{Antoine MOUZARD and Immanuel ZACHHUBER}
\date{}

\begin{document}

\maketitle
\abstract{In this paper, we prove existence and uniqueness of energy solutions for nonlinear Schrödinger equations with a multiplicative white noise on $\IR^d$ with $d\le 3$. We rely on an exponential transform and conserved quantities for existence of energy solutions. Using paracontrolled calculus, we prove Strichartz inequalities which encode the dispersive properties of the solutions. This allows to obtain local well-posedness for low regularity solutions and uniqueness of energy solutions for various equations. In particular, our results are the first results of propagation without loss of both regularity and localization for such equations in full space as well as the first results on $\IR^3$ for such singular dispersive SPDEs. We are also obtain local well-posedness in two dimensions for deterministic initial data.}
% \tableofcontents
\vspace{0.5cm}

\section{Introduction}

\medskip

We study the nonlinear Schrödinger equation with multiplicative white noise
\begin{equation}\label{NLS}\tag{NLS}
i\partial_tu=-\Delta u+u\xi+|u|^{m-1}u
\end{equation}
on $\IR^d$ for $d\le3$ and $m\ge1$ and the Hartree equation
\begin{equation}\label{Hartree}\tag{Hartree}
i\partial_tu=-\Delta u+u\xi+u\cdot V_\beta*|u|^2
\end{equation}
on $\IR^3$ with nonnegative symmetric $V_\beta\in W^{\beta,1}$ for $\beta>0$.  The noise $\xi$ is a spatial Gaussian white noise with covariance $\IE[\xi(x)\xi(y)]=\delta_0(x-y)$ on $\IR^d$. The difficulty lies in the irregularity of the noise as well as its behavior at infinity since $\xi\in C_{-\delta}^{-\frac{d}{2}-\kappa}$ for any $\delta,\kappa>0$, that is growing at infinity slower than any polynomial $\langle x\rangle^\delta$ with negative Hölder regularity. In dimension $d\ge2$, the noise is too irregular for the equation to make sense by itself and its resolution involves a probabilistic renormalization procedure. The resolution of such singular dispersive SPDEs have been recently achieved following the progess on their parabolic counterparts with regularity structures \cite{Hai14} and paracontrolled calculus \cite{GIP}. For dispersive PDEs, the lack of strong regularization of the Schrödinger group in comparison to the heat semigroup makes this study very different as regularity is key for singular SPDEs.

\medskip

The first resolution of dispersive singular SPDEs was done by by Debussche and Weber \cite{DW} on $\IT^2$ for the cubic \eqref{NLS} equation based on an exponential transform with the new unknown
\begin{equation}
u=e^Xv
\end{equation}
with suitable random field $X$ following Hairer and Labbé \cite{HairerLabbe15} on the parabolic Anderson model on $\IR^2$. To solve the equation, one considers a regularization $(\xi_\eps)_{\eps>0}$ of the potential and relies on uniform bounds on the mass and energy after a renormalization procedure to construct solutions in the limit. This gives existence of solutions while one proves uniqueness with strong enough a priori bounds on the solution, again obtained via bounds on the regularized solutions uniform in $\eps>0$. This was improved in various way on $\IT^2$ by Tzvetkov and Visciglia \cite{TzvetkovVisciglia23,TzvetkovVisciglia23bis} and adapted to $\IR^2$ in different context, see \cite{ChauleurMouzard23,DRTV24,DebusscheMartin19}. The main drawback of this approach based on an exponential transform is the loss of regularity of the solution with respect to the initial data, as well as loss of spatial localization on $\IR^2$. In the compact case, an approach based on paracontrolled calculus can also be used with the construction of the Anderson Hamiltonian
\begin{equation}
\CH=-\Delta+\xi
\end{equation}
as first done by Gubinelli, Ugurcan and Zachhuber \cite{GUZ} on $\IT^d$ with $d\in\{2,3\}$. This operator can be constructed with a random domain $\CD(\CH)$, it has a compact resolvent hence a basis of eigenfunctions with bounded from below eigenvalues. This allows to define the Schrödinger group $e^{it\CH}$ and one can solve the cubic \eqref{NLS} equation with the associated mild formulation. This gives global well-posedness for initial data $u_0\in\CD(\CH)$ as well as the existence of energy solution for initial data in the form domain $\CD(\sqrt{\CH})$. Uniqueness for such equations in the case without noise requires Strichartz inequalities which encodes the dispersive regularization properties of the Schrödinger flow. For the Anderson Hamiltonian, this was obtained by the authors in the previous work \cite{MZ} on any compact surfaces, using the construction by the first author in \cite{Mouzard} based on second order paracontrolled calculus. For example, the operator satisfies the Strichartz inequalities
\begin{equation}
\|e^{it\CH}u_0\|_{L^p([0,T],L^q(M))}\lesssim\|u_0\|_{H^{\frac{1}{p}+\kappa}}
\end{equation}
for $(p,q)$ such that $\frac{1}{p}+\frac{1}{q}=\frac{1}{2}$ and any $\kappa>0$ on a compact surfaces without boundary. This allows to obtain low regularity solutions, that is local well-posedess for the cubic \eqref{NLS} equation for deterministic initial data $u_0\in H^\sigma$ with $\frac{1}{2}<\sigma<1$ from which one recovers uniqueness of energy solutions. In the deterministic case setting $\xi$ to zero, this result was obtained by Burq, Gérard and Tzvetkov \cite{BGT} without the $\kappa$ loss of derivatives which is optimal for general compact surfaces. The approach based on the Anderson Hamiltonian can also be considered via the exponential transform on compact spaces since it gives the form domain of the operator, see for example \cite{MatsudaZuijlen22,MouzardOuhabaz23}, however paracontrolled calculus seems necessary to obtain Strichartz inequalities. One of the main differences also lies in the well-posedness result obtained, exponential transform induces a loss of regularity at positive time as one can see in \cite{DW} while the use of the Anderson Hamiltonian gives a random space that is propagated by the dynamic. In full space, the study of this random operator and its spectral properties are too involved for the moment to help the study of the associated equation, for example its spectrum is $\IR$ in $\IR^d$ for $d\le3$, see \cite{HsuLabbe25,Ueki25} and references therein. 

\medskip

In this work, we prove global well-posedness of energy solutions without loss of regularity or localization on $\IR^d$ with $d\le 3$. In two dimensions, we consider \eqref{NLS} for general $m\ge2$. In three dimensions, we are restricted to the \eqref{Hartree} equation for $\beta>\frac{1}{2}$. We obtain existence for more general parameters such as \eqref{NLS} for $1\le m<6$ or \eqref{Hartree} for $\beta>0$ on $\IR^3$, the restriction is needed to prove uniqueness of solutions. Our approach being based on an exponential transform, we rely on uniform bounds with respect to $\eps>0$ of a suitable renormalized equation and compactness to obtain existence of solutions in the limit with $X$ a well-chosen random field.

\medskip 

\begin{theorem*}
For $\mu\in(0,1]$ and any initial data $u_0\in e^{-X}(L_\mu^2(\IR^d)\cap H^1(\IR^d))$, there exists energy solutions in $C\big(\IR,e^{-X}(L_\mu^2\cap H^1)\big)$ to \eqref{NLS} equation on $\IR^2$ for $m\ge1$, to \eqref{NLS} equation on $\IR^3$ for $1\le m<6$ and to \eqref{Hartree} equation on $\IR^3$ for $\beta>0$. The solution is unique for \eqref{NLS} equation on $\IR^2$ for $m\ge1$ and \eqref{Hartree} equation on $\IR^3$ for $\beta>\frac{1}{2}$.
\end{theorem*}

\medskip

Our main contribution to this argument is threefold. Firstly, we prove that with a suitable definition of energy solutions, one can avoid the loss of regularity for solutions at positive time. In particular, this applies to the previous studied framework such as $\IT^2$ by Debussche and Weber \cite{DW} to get existence of energy solutions. Secondly, we consider a new random field $X$ on full space which does not grow at infinity while still play the role of removing the singularity from local irregularity. Global strong solutions for \eqref{NLS} on $\IR^2$ for any $m\ge2$ were recently obtained in \cite{DRTV24} with a loss of both regularity and localization, improving the previous work \cite{DebusscheMartin19}. In the present work, we obtain global well-posedness without any loss. In three dimensions, the only two papers up to our knowledge are on the torus $\IT^3$, see \cite{GUZ} for \eqref{NLS} and \cite{DeVecchiJiZachhuber25} for \eqref{Hartree}. Finally, we obtain Strichartz inequalities for the linear equation on $\IR^d$ for $d\in\{2,3\}$ using paracontrolled calculus. We follow the perturbative argument of our previous work \cite{MZ} using paracontrolled calculus with a modified operator $\CH^\sharp=\Gamma^{-1}\CH\Gamma$ which is a better behaved perturbation of the Laplacian in comparison to $\CH$. The map $\Gamma$ is a random map defined in Appendix \ref{sec:construction-AH} and the only ingredient needed from paracontrolled calculus, see \cite{DebusscheMouzard24} for a similar argument for a deterministic rough potential. This allows to obtain uniqueness of energy solutions as well as local well-posedness for low regularity initial data.

\medskip 

\begin{theorem*}
For any $m\ge2$, \eqref{NLS} is locally well-posed in $L_\sigma^2(\IR^2)\cap H^\sigma(\IR^2)$ for $\sigma\in(\sigma_m,1)$ with $\sigma_m=1-\frac{1}{m-1}$. For any $\beta>\frac{1}{2}$, there exists $\sigma(\beta)<1$ such that equation \eqref{Hartree} is locally well-posed in $e^{-X}(L_\sigma^2(\IR^3)\cap \Gamma H^\sigma(\IR^3))$ for $\sigma\in(\sigma_\beta,1)$. In particular, the solution depends continuously on the initial data.
\end{theorem*}

\medskip

As mentionned, the only results in three dimensions were obtained in \cite{GUZ,DeVecchiJiZachhuber25} on $\IT^3$. Local well-posdness for \eqref{NLS} follows from \cite{GUZ} while the authors proved global well-posedness for \eqref{Hartree} for $\beta>\frac{17}{20}$, see Sections $4$ and $5$ in \cite{DeVecchiJiZachhuber25}. We get here the improved condition $\beta>\frac{1}{2}$ with respect to \cite{DeVecchiJiZachhuber25}, this is due to the use of better dispersive properties for the deterministic Schrödinger flow $e^{it\Delta}$ on $\IR^3$ than on $\IT^3$. While this is natural to expect, this is the first result of this type for such singular dispersive PDEs. Indeed, singular stochastic PDEs are harder to solve on full space due to the invariance by translation of the noise. In order to prove the previous result, we follow a classical fixed point argument which requires to work with a space stable by the flow. This is only possible because we are able to solve the equation without loss of regularization and localization and relies crucially on our spatial localization argument. One of the main interest of the previous result is that we are able to consider deterministic initial data for \eqref{NLS} on $\IR^2$. On compact surfaces, this result is only known for low regularity initial data with our previous work \cite{MZ}, we prove here that this is also possible on $\IR^2$ given polynomial decay at infinity.

\medskip

Since the work by Debussche and Weber \cite{DW} on $\IT^2$, several works studied dispersive singular SPDEs on full space. Debussche and Martin \cite{DebusscheMartin19} adapted the argument to consider \eqref{NLS} for $m<2$ on $\IR^2$ for strong solutions with loss of regularity and localisation. This was generalised to any $m\ge1$ in \cite{DRTV24} with additional arguments from \cite{TzvetkovVisciglia23bis,TzvetkovVisciglia23} using modified energy methods. Remark that they prove global well-posedness in the sense of existence and uniquenss of a global solution, without propagation of a space of initial data for $t\in\IR$. In a different direction, Chauleur and Mouzard \cite{ChauleurMouzard23} considered a logarithmic nonlinearity which behave nicely with the exponential transform despite not being Lipschitz. More recently, we mention the recent work \cite{LiuTzvetkov2026} where the convergence of solutions on large torus to the solutions on full space is studied. In a different direction, Mackowiak considered \eqref{NLS} with an additional confining harmonic potential on $\IR^2$, see \cite{Mackowiak25,Mackowiak2025}. To the best of our knowledge, this covers all results on $\IR^2$ while no results were known on $\IR^3$.

\bigskip

We now describe the main arguments of the present work. In this context of singular dispersive SPDEs one needs to define carefully a notion of energy solutions, we consider \eqref{NLS} equation in two dimensions in this introduction for the sake of clarity. For the usual equation without noise, energy solutions are defined for $u_0\in H^1$ satisfying the equation in the dual space $(H^1)^*=H^{-1}$. In our case, $u_0\in e^X(H^1\cap L_\mu^2)$ gives a solution in the dual space of $e^X(H^1\cap L_\mu^2)$ which does not contain smooth functions. The new variable $u=e^Xv$ gives the formal equation
\begin{equation}
ie^X\partial_tv=-e^X\Delta v-2e^X\nabla X\cdot\nabla v+e^X(\xi-\Delta X-|\nabla X|^2)v+e^{(m-1)X}|v|^{m-1}v
\end{equation}
which usually motivates the choice $\Delta X=\xi$ to cancel the roughest part of the equation. This is the path followed by Debussche and Weber \cite{DW}, they considered this new equation on $v$ with $\Wick{|\nabla X|^2}$ interpreted as a Wick product since $\xi\in\CC^{-1-\kappa}(\IT^2)$ hence $X\in\CC^{1-\kappa}(\IT^2)$, this is the probabilistic renormalization of the singular SPDE. For the product term $\nabla X\cdot\nabla v$ to be well-defined, one needs regularity higher than one thus motivating them to consider strong solution with better regularity, that is $v_0\in e^XH^2(\IT^2)$. To get a notion of energy solutions, the equation should be satisfied in the dual space of $e^XH^1(\IT^2)$ hence tested against $e^X\varphi$ for $\varphi\in H^1(\IT^2)$ to get
\begin{equation}
\big\langle i\partial_t ve^X,e^X\varphi\big\rangle=\big\langle-e^X\Delta v-e^X2\nabla X\cdot\nabla v-e^X\Wick{|\nabla X|^2}v+e^{(m-1)X}|v|^{m-2}v,e^X\varphi\big\rangle
\end{equation}
as weak formulation of the equation. This amounts to a multiplication of the equation by $e^X$ which gives the equation
\begin{equation}
ie^{2X}\partial_tv=-\big(\nabla e^{2X}\nabla\big)v-e^{2X}\Wick{|\nabla X|^2}v+e^{mX}|v|^{m-2}v
\end{equation}
in $H^{-1}(\IT^2)$ where $v\in H^1(\IT^2)$ is enough for all the terms to make sense. Since $X\in\CC^{1-\kappa}(\IT^2)$, multiplying the equation by $e^{-2X}$ is a singular product which corresponds to the singular product $\nabla X\cdot\nabla v$ for $v\in H^1(\IT^2)$ as in \cite{DW}. Remark that this computation is similar to the construction of the Anderson Hamiltonian on $\IT^d$ for $d\in\{2,3\}$ via its quadratic form from \cite{MouzardOuhabaz23}. In the compact case of the torus, the energy of the equation is formally
\begin{equation}
\langle\CH u,u\rangle=\langle\CH e^Xv,e^Xv\rangle=\langle e^X\CH e^Xv,v\rangle
\end{equation}
for any $v\in H^1$ where $\CH=-\Delta+\xi$ corresponds to the Anderson Hamiltonian. Using the notation $Tv=e^{X}v$, the equation on $v$ is
\begin{equation}
i\partial_tv=T^{-1}\CH T v+T^{-1}(|T v|^{m-2}T v)
\end{equation}
where $T^{-1}u=e^{-X}u$ which involves the singular gradient product for energy solution while the equation
\begin{equation}
iT^*T \partial_tv=T^*\CH T v+T^*(|T v|^{m-2}T v)
\end{equation}
with $T^*=T$ is the equation to consider for energy solution. Since $T$ is not a unitary transform, this yields two different equations, one of our contributions is to identify the second one as a natural notion of energy solutions. For example, this allows to get existence of energy solutions using a compactness argument relying only on uniform bounds for the energy with the exponential transform without loss of derivatives in the case of the two dimensional torus with $u_0\in e^XH^1(\IT^2)$. For uniqueness, Strichartz inequalities are needed which a priori requires paracontrolled calculus. This is the strategy we follow to get global well-posedness for \eqref{NLS} and \eqref{Hartree} on $\IR^d$ with $d\in\{2,3\}$ without any loss of regularity. While we do not rely on the spectral properties of the Anderson operator in full space, we use the notation
\begin{equation}
\CH u=-\Delta u+u\xi
\end{equation}
as the computation are similar to the ones on $\IT^d$.

\medskip

On full space, there is an additional problem of growth at infinity and previous work considered localised initial data. In addition to the loss of regularity, the use of the exponential transform induces a loss of localization at positive time due to the growth at infinity of $e^X$ and $e^{-X}$. To get around this problem, we consider a new random field $X$ such that $e^X,e^{-X}\in L^\infty$ while still cancelling the roughest part of the equation. Following a decomposition introduced by Gubinelli and Hofmanovà in \cite{GubinelliHofmanova19}, one can split the white noise $\xi\in\CC_{-\delta}^{-\frac{d}{2}-\kappa}(\IR^d)$ in two terms as
\begin{equation}
\xi=\xi^\le+\xi^>
\end{equation}
with $\xi^\le\in L_{-\delta}^\infty(\IR^d)$ and $\xi^>\in\CC_\delta^{-\frac{d}{2}-\kappa}(\IR^d)$ for any $\delta,\kappa>0$ using Lemma \ref{Lem:LocalizationProjectors}. The first term encodes the growth at infinity while being more regular and the second term has the same local regularity with decay at infinity. Considering for example $\Delta X=\xi^>$ gives the equation
\begin{equation}
ie^{2X}\partial_tv=-\big(\nabla e^{2X}\nabla\big)v+e^{2X}(\xi^\le-\Wick{|\nabla X|^2})v+e^{mX}|v|^{m-2}v
\end{equation}
where the product with the growing smooth term $\xi^\le$ is not singular. The important property is that the norms of the spaces $L^2(\IR^d,e^{qX}\drm x)$ and $L^2(\IR^d,\drm x)$ are equivalent for any $q\in\IR$ thus it avoids loss of localization between $u$ and $v$. Given this new localized rough field, we adapt the argument from Debussche and Weber to the full space without any loss which give the existence of energy solutions. Note that in three dimensions, the renormalization procedure is more involved, as done on $\IT^3$ in \cite{GUZ}. To prove uniqueness, we obtain Strichartz inequalities for the linear equation
\begin{equation}
ie^{2X}\partial_tv=e^X\CH e^{X}v
\end{equation} 
with $v_0\in L_\mu^2\cap H^1$. We follow a perturbative argument using the result for the linear flow $e^{it\Delta}$, the idea being to consider a mild formulation of the linear equation as
\begin{equation}
% v_t=e^{it\Delta}v_0-i\int_0^te^{i(t-s)\Delta}(2\nabla X\cdot\nabla v_s-Yv_s)\drm s
% v_t^\sharp=e^{it\Delta}v_0^\sharp-i\int_0^te^{i(t-s)\Delta}(\CH^\sharp+\Delta)v_s^\sharp\drm s
v_t=e^{it\Delta}v_0-i\int_0^te^{i(t-s)\Delta}(e^{-X}\CH e^X+\Delta)v_s\drm s
\end{equation}
with the use of spectral projections both for the solution and its initial condition. However $e^{-X}\CH e^X+\Delta$ is not well-behaved enough for this proof to work and we need refined tools, as it was already the case on compact surfaces, see \cite{MZ}. This argument also appears in \cite{DRTV24} on $\IR^2$ where the authors obtained Strichartz inequalities for the regularized equation that blow up logarithmically as $\eps>0$ goes to $0$ which still alllows them to obtain uniqueness. This does not yield Strichartz inequalities for the solution in the limit and does not apply to $\IR^3$ as the blow up is polynomial. In this work, we use paracontrolled calculus to prove Strichartz inequalities for the limiting equation in two and three dimensions. The main tool is a random transformation $\Gamma$ such that 
\begin{equation}
\CH^\sharp=(\Gamma e^X)^{-1}\CH(\Gamma e^X)
\end{equation} 
is a better behaved perturbation of the Laplacian which was first used in \cite{GUZ}, see also \cite{EulryMouzard25,Mackowiak25,Mouzard,DeVecchiJiZachhuber25} and references therein. The mild formulation
\begin{equation}
% v_t=e^{it\Delta}v_0-i\int_0^te^{i(t-s)\Delta}(2\nabla X\cdot\nabla v_s-Yv_s)\drm s
% v_t^\sharp=e^{it\Delta}v_0^\sharp-i\int_0^te^{i(t-s)\Delta}(\CH^\sharp+\Delta)v_s^\sharp\drm s
v_t^\sharp=e^{it\Delta}v_0^\sharp-i\int_0^te^{i(t-s)\Delta}(\CH^\sharp+\Delta)v_s^\sharp\drm s
\end{equation}
with $v=\Gamma v^\sharp$ allows to obtain Strichartz inequalities, see Theorem \ref{thm:Strichartz} for a precise statement. This allows to obtain low regularity solutions using a fixed point formulation as well as uniqueness of energy solutions. While this was done on compact spaces in \cite{MZ}, there is the additionnal difficulty of growth at infinity of the potential. Indeed, the previous mild formulation can not be used a priori since the right hand side induces a arbitrary small loss of localization thus preventing the closure of the fixed point. While a mild formulation with respect to the Laplacian is crucial to obtain Strichartz inequalities for the linear equation, we consider the mild formulation associated to $\CH$ to solve the associated nonlinear equations. In particular, the linear equation
\begin{equation}
ie^{2X}\partial_tv=e^X\CH e^Xv
\end{equation}
defined a linear flow $S_t$ a priori defined on $L_1^2(\IR^d)\cap H^1(\IR^d)$ with our notion of energy solutions. Since we do not have loss of localization in our construction using energy, this allows to avoid the previous loss of localization in the mild formulation with the local well-posedness result mentioned above.

\medskip

The plan of the paper is as follows. In section \ref{Sec:propagation_1d}, we prove a propagation of regularity in one dimension for \eqref{NLS} using the new random field $X$ to avoid the loss of localization. In section \ref{Sec:existence}, we construct energy solutions with a compactness argument without loss of regularity or localization for \eqref{NLS} on $\IR^d$ for $d\in\{2,3\}$ and \eqref{Hartree} on $\IR^3$. We prove Strichartz inequalities in section \ref{Sec:StrongandStrichartz}, this relies on strong solutions for the linear equations. In section \ref{Sec:lowregsolutions}, we prove local well-posedness using a fixed point formulation as well as uniqueness of energy solutions. The needed toolbox of function spaces are recalled in Appendix \ref{Sec:FunctionalSpaces} while Appendix \ref{Sec:stochastic_bounds} gives the different bounds on stochastic objects, including renormalization in two and three dimensions with the localization in full space. The construction of $\CH$ using an exponential transform and then paracontrolled calculus is done in Appendix \ref{sec:construction-AH}.

\medskip

{\bf Acknoledgments :} The second author acknowledges funding by the Deutsche Forschungsgemeinschaft (DFG, German Research Foundation) CRC/TRR 388 Rough Analysis, Stochastic Dynamics and Related Fields Project ID 516748464.

\section{Energy solutions and propagation of regularity in 1D}\label{Sec:propagation_1d}

In one dimension, the equation is not singular and one can get global energy solutions with uniform energy bounds for the solution with regularized noise and pass to the limit. The initial data must have polynomial decay for the energy to be bounded due to the growth of the noise, we are able to propagate $H^1(\IR)\cap L_\mu^2(\IR)$ for $0<\mu\le1$. However one can not propagate strong solutions in $H^2(\IR)\cap L_\mu^2(\IR)$ because of the local irregularity of the noise. We prove that a space of well-prepared data can be propagated using the exponential transform, this was done for example in section $5$ of \cite{ChauleurMouzard23} for the stochastic logarithmic NLS equation. In particular, we already improve this result using the spatial localization defined in Appendix \ref{Sec:stochastic_bounds} from \cite{GubinelliHofmanova19}. We prove the propagation of the random space
\begin{equation}
\CD_\mu^2(\IR):=e^{X}\big(H^2(\IR)\cap L_\mu^2(\R)\big)
\end{equation}
for $0<\mu\le 1$ with no loss of localization or regularity where $X$ is a suitable Gaussian field depending on $\xi$. The main argument is to consider $u=e^Xv$ which satisfies the equation
\begin{equation}
i\partial_tv=-\Delta v-2\nabla X\cdot\nabla v+(\xi-\Delta X-|\nabla X|^2)v+e^{-(m-1)X}|v|^{m-1}v
\end{equation}
with $v_0\in H^2(\IR)\cap L_\mu^2(\IR)$. In particular, the solution belongs to $\CH^{\frac{3}{2}-\kappa}(\R)\cap L_{\mu}^2(\R)$ for any $\kappa>0$ using the regularity of $X$.

\medskip

\begin{theorem}\label{thm:WP1D}
For $0<\mu\le1$, there exists a unique solution $u \in \mathcal{C}\big(\R,H^1(\R)\cap L_{\mu}^2(\R)\big)$ to \eqref{NLS} with initial data $u_0\in H^1(\R)\cap L_{\mu}^2(\R)$. Given two solutions $u,u'$ with respective initial data $u_0,u_0'\in H^1(\R)\cap L_{\mu}^2(\R)$ and $T>0$, we have
\begin{equation}
\sup_{t\in[-T,T]}\|u(t)-u'(t)\|_{L^2(\IR)}\le C_{T,u_0,u_0'} \|u_0-u_0'\|_{L^2(\IR)}
\end{equation}
with $C_{T,u_0,u_0'}>0$. If moreover $u_0\in\CD_{\mu}^2$, then $u\in\CC(\R,\CD_\mu^2)$.
\end{theorem}

\medskip 

\begin{proof}
Let $u_\eps$ be the solution to the regularized equation
\begin{equation}
i\partial_tu_\eps=-\Delta u_\eps+u_\eps\xi_\eps+|u_\eps|^{m-1}u_\eps
\end{equation}
with initial data $u_\eps(0)=u_0$. It has conserved mass
\begin{equation}
\int_\IR|u_\eps(t,x)|^2\drm x
\end{equation}
and energy
\begin{equation}
\CE_\eps(t)=\frac{1}{2}\int_\IR|\nabla u_\eps(t,x)|^2\drm x+\frac{1}{2}\int_\IR|u_\eps(t,x)|^2\xi_\eps(x)\drm x+\frac{1}{m+1}\int_\IR|u_\eps(t,x)|^{m+1}\drm x
\end{equation}
for any $\eps>0$. Since the equation is defocusing, this gives
\begin{equation}
\int_\IR|\nabla u_\eps(t,x)|^2\drm x\le \CE_\eps(0)-\int_\IR|u_\eps(t,x)|^2\xi_\eps(x)\drm x
\end{equation} 
hence
\begin{align}
\int_\IR|\nabla u_\eps(t,x)|^2\drm x&\lesssim \CE_\eps(0)+\|u_\eps(t)^2\|_{\CW_{2\delta}^{\frac{1}{2}+\kappa,1}}\|\xi_\eps\|_{C_{-2\delta}^{-\frac{1}{2}-\kappa}}\\
&\lesssim \CE_\eps(0)+\|u_\eps(t)\|_{H_{\delta}^{\frac{1}{2}+2\kappa}}^2\|\xi_\eps\|_{C_{-2\delta}^{-\frac{1}{2}-\kappa}}\\
&\lesssim \CE_\eps(0)+\|u_\eps(t)\|_{L_{\frac{\delta}{\frac{1}{2}-2\kappa}}^2}^{1-4\kappa}\|u_\eps(t)\|_{H^1}^{1+4\kappa}\|\xi_\eps\|_{C_{-2\delta}^{-\frac{1}{2}-\kappa}}
\end{align}
where $\CW_{2\delta}^{\frac{1}{2}+\kappa,1}=\CB_{1,1,2\delta}^{\frac{1}{2}+\kappa}$ for $\delta,\kappa>0$ arbitrary small using Lemmas \ref{interpolation_estimate_Besov}, \ref{product_estimate_Besov} and \ref{duality_estimate_Besov}. For the weighted norm, we use the equation to get
\begin{align}
\frac{\drm}{\drm t}\|u_\eps(t)\|_{L_{\eta}^2}^2&=-\text{Im}\Big(2\eta\int_\IR\frac{x}{\langle x\rangle^{2\eta-2}}u_\eps(t,x)\nabla\overline u_\eps(t,x)\drm x\Big)\\
&\lesssim\|u_\eps(t)\|_{L^2}\|\nabla u_\eps(t)\|_{L_{2\eta-1}^2}\\ 
&\lesssim\|u_0\|_{L^2}\|u_\eps(t)\|_{H^1}
\end{align}
for $2\eta-1\le0$ using the conservation of mass. We get
\begin{equation}
\sup_{t\in I}\|u_\eps(t)\|_{L_{\eta}^2}^2\lesssim\|u_0\|_{L_{\eta}^2}^2+|I|\cdot\|u_0\|_{L^2}\sup_{t\in I}\|u_\eps(t)\|_{H^1}
\end{equation}
for any time interval $I$ containing $0$ with $\eta\le\frac{1}{2}\wedge\mu$. Together with the previous energy bound, this yields
\begin{equation}
\sup_{t\in I}\int_\IR|\nabla u_\eps(t,x)|^2\drm x\lesssim \CE_\eps(0)+|I|\sup_{t\in I}\|u_\eps(t)\|_{H^1}^{\frac{3}{2}+2\kappa}
\end{equation}
using that $\xi_\eps$ in uniformly bounded in $C_{-2\delta}^{-\frac{1}{2}-\kappa}$. Since $u_0\in H^1(\IR)\cap L_\mu^2(\IR)$, the energy of the initial data $\CE_\eps(0)$ is uniformly bounded and
\begin{equation}
\sup_{\eps>0}\ \sup_{t\in I}\|u_\eps(t)\|_{H^1}<\infty
\end{equation}
for any finite interval $I\subset\IR$. Up to extraction using compactness, this yields the existence of weak solution $u\in\CC(\IR,H^1\cap L_\nu^2)$ satisfying the equation in $H^{-1}\cap L_{-\mu}^2$, see for example Ginibre and Velo \cite{GinibreVelo85} or Proposition $4.1$ in \cite{ChauleurMouzard23}. To propagate localization up to $\mu=1$, we write
\begin{align}
\frac{\drm}{\drm t}\|u_\eps(t)\|_{L_{\eta}^2}^2&=-\text{Im}\Big(2\eta\int_\IR\frac{x}{\langle x\rangle^{2\eta-2}}u_\eps(t,x)\nabla\overline u_\eps(t,x)\drm x\Big)\\
&\lesssim\|u_\eps(t)\|_{L_{\frac{1}{2}}^2}\|\nabla u_\eps(t)\|_{L_{2\eta-\frac{3}{2}}^2}\\ 
&\lesssim\|u_0\|_{L_{\frac{1}{2}}^2}\|u_\eps(t)\|_{H^1}
\end{align}
for $\eta\le\frac{3}{4}$ using that one has already propage $\eta=\frac{1}{2}$, one can go up to $\eta=1$ with iteration of this argument. Uniqueness follows directly using $H^1(\IR)\hookrightarrow L^\infty(\IR)$. Indeed, for $u$ and $u'$ two solutions with respective initial data $u_0$ and $u_0'$, we have
\begin{equation}
i\partial_tw=-\Delta w+w\xi+|u|^{m-1}u-|u'|^{m-1}u'
\end{equation}
where $w=u-u'$. We get
\begin{align}
\frac{\drm}{\drm t}\|w(t)\|_{L^2}^2&\lesssim\text{Im}\int_{\IR}\big(|u(t,x)|^{m-1}u(t,x)-|u'(t,x)|^{m-1}u'(t,x)\big)w(t,x)\drm x\\ 
&\lesssim\int_{\IR}(1+|u(t,x)|^{m-1}+|u'(t,x)|^{m-1})|w(t,x)|^2\drm x\\ 
&\lesssim\big(1+\|u(t)\|_{L^\infty}^{m-1}+\|u'(t)\|_{L^\infty}^{m-1}\big)\|w(t)\|_{L^2}^2\\
&\lesssim\big(1+\|u(t)\|_{H^1}^{m-1}+\|u'(t)\|_{H^1}^{m-1}\big)\|w(t)\|_{L^2}^2
\end{align}
which yields
\begin{equation}
\|u(t)-u'(t)\|_{L^2}^2\lesssim e^{\int_0^t\big(1+\|u(s)\|_{H^1}^{m-1}+\|u'(s)\|_{H^1}^{m-1}\big)\drm s}\|u_0-u_0'\|_{L^2}^2
\end{equation}
with Gronwall's lemma hence uniqueness and dependance on the initial data. Finally, we prove the propagation of the random space
\begin{equation}
\CD_{\mu}^2=e^X(L_{\mu}^2\cap H^2)
\end{equation}
where $X\in C_\delta^{\frac{3}{2}-\kappa}$ such that
\begin{equation}
Y=\xi-\Delta X-|\nabla X|^2\in\CC_{-\delta}^{\frac{1}{2}-\kappa}
\end{equation}
for any $\kappa,\delta>0$ and $e^{qX}\in L^\infty$ for any $q\in\IR$, see Appendix \ref{Sec:stochastic_bounds} for details on localization. Since $u(t)\in L_\mu^2$ and $e^{-X}\in L^\infty$, we have
\begin{equation}
v(t)=e^{-X}u(t)\in L_\mu^2
\end{equation}
and it only remains to prove that $v(t)=e^{-X}u(t)\in H^2$. It satisfies the transformed equation
\begin{equation}
i\partial_tv=-\Delta v-2\nabla X\cdot\nabla v+Yv+e^{(m-1)X}|v|^{m-1}v
\end{equation}
with initial data $v_0\in L_{\mu}^2\cap H^2$. The time derivative $w=\partial_tv$ satisfies the equation
\begin{equation}
i\partial_tw=-\Delta w-2\nabla X\cdot\nabla w+Yw+e^{(m-1)X}\big(\frac{m-1}{2}+1\big
)|v|^{m-1}w+e^{(m-1)X}\frac{m-1}{2}|v|^{m-3}v^2\overline{w}
\end{equation}
which gives
\begin{align}
\frac{\drm}{\drm t}\int_\IR|w(t,x)|^2e^{2X(x)}\drm x&=\frac{m-1}{2}\textrm{Im}\Big(\int_\IR|v(t,x)|^{m-3}v(t,x)^2\overline{w}(t,x)e^{(m-1)X(x)}\drm x\Big)\\ 
&\lesssim\|v^{m-1}(t)\|_{L^2}\|w(t)e^{X}\|_{L^2}\\
&\lesssim\|v(t)\|_{H^1}^{m-1}\|w(t)e^{X}\|_{L^2}
\end{align}
using that $e^{(m-2)X}\in L^\infty$. Gronwall's lemma gives
\begin{equation}
\|w(t)e^{X}\|_{L^2}\le e^{ct}\|w(0)e^{X}\|_{L^2}
\end{equation}
for a positive constant $c>0$ depending on the initial data. Using the equation on $v$ at time $t=0$, we have
\begin{equation}
iw(0)=-\Delta v_0-2\nabla X\cdot\nabla v_0+Yv_0+e^{(m-1)X}|v_0|^{m-1}v_0
\end{equation}
which belongs to $L^2$ for $v_0\in L_{\mu}^2\cap H^2$ with $\mu>0$ since $X\in C_\delta^{\frac{3}{2}-\kappa}$ and $Y\in C_{-\delta}^{\frac{1}{2}-\kappa}$. This implies that $w(t)e^X\in L^2$ for any $t>0$ thus $w(t)\in L^2$ since $e^{-X}\in L^\infty$. Again, the equation on $v$ gives
\begin{equation}
\Delta v=-i\partial_tv-2\nabla X\nabla v+Yv+e^{(m-1)X}|v|^{m-1}v
\end{equation}
hence $\Delta v(t)\in L^2$ which completes the proof of propagation of regularity for $v$.
\end{proof}

\begin{remark}
We considered here the defocusing equation but similar result also holds for subquintic focusing NLS with $m<6$ or $m=6$ with small initial data using the same interpolation argument for the $L^p$ norm between $L^2$ and $H^1$ in the conservation of energy.
\end{remark}

\section{Existence of energy solutions in higher dimensions}\label{Sec:existence}

We define a notion of energy solution for the equation
\begin{equation}\label{GeneralEquation}
i\partial_tu=-\Delta u+u\xi+f(u)
\end{equation}
on $\IR^d$ with $d\in\{2,3\}$ and $f(u)$ the nonlinearity. The initial data $u_0$ is taken in the random space 
\begin{equation}
\CD_\mu^1(\IR^d):=e^{X}\big(H^1(\IR^d)\cap L_\mu^2(\R^d)\big)
\end{equation}
for $0<\mu\le1$ and a suitable random field $X$ depending on $\xi$ and the dimension. As for $d=1$, one can consider $u=e^Xv$ and multiply the equation by $e^X$ to formally get
\begin{equation}
ie^{2X}\partial_tv=-(\nabla e^{2X}\nabla)v+e^{2X}(\xi-\Delta X-|\nabla X|^2)v-e^Xf(e^Xv)
\end{equation}
with initial data $v_0\in H^1(\IR^d)\cap L_\mu^2(\IR^d)$. For $d\in\{2,3\}$, this falls in the range of singular dispersive SPDEs and requires a probabilistic renormalization procedure. Indeed, $\nabla X$ is a distribution of Hölder regularity $\gamma<1-\frac{d}{2}$ thus its square $|\nabla X|^2$ is ill-defined. The equation makes sense for $\xi_\eps$ a smooth regularization of the noise and the renormalization procedure amounts to the construction
\begin{equation}
Y=\lim_{\eps\to 0}(\xi_\eps-\Delta X_\eps-|\nabla X_\eps|^2-c_\eps)=:\lim_{\eps\to 0}Y_\eps
\end{equation}
for a suitable random field $Y$ and $c_\eps$ a diverging constant. This idea was used on $\IT^2$ by Debussche and Weber \cite{DW} with $\Delta X=\xi$ up to the mode zero, see also \cite{ChauleurMouzard23,DRTV24,DebusscheMartin19,TzvetkovVisciglia23,TzvetkovVisciglia23bis} for similar constructions on $\IT^2$ and $\IR^2$. On full space, this choice induces a growth at infinity of the random field $Y$ which causes a loss in localization at positive time with respect to the initial data. We consider a different random field $X$ using a spatial localization, see Appendix \ref{Sec:stochastic_bounds} where we prove the convergence of $Y_\eps$ to a limit $Y$ in $\CC_{-\delta}^{\frac{2-d}{2}-\kappa}$ for any $\delta,\kappa>0$. The main idea is that one can cancel the local irregularity of the rough distributions without including the growth at infinity in the field $X$. The equation
\begin{equation}
ie^{2X_\eps}\partial_tv_\eps=-(\nabla e^{2X_\eps}\nabla)v_\eps+e^{2X_\eps}(\xi_\eps-\Delta X_\eps-|\nabla X_\eps|^2)v_\eps-e^{X_\eps}f(e^{X_\eps}v)
\end{equation}
then converges as $\eps$ goes to $0$ to
\begin{equation}
ie^{2X}\partial_tv=-(\nabla e^{2X}\nabla)v+e^{2X}Yv-e^Xf(e^Xv)
\end{equation}
where $e^{2X}Yv$ makes sense for $v\in L_\mu^2(\IR^d)\cap H^1(\IR^d)$. If one multiply the initial equation by $e^{-X}$ instead of $e^X$, this gives the formal equation
\begin{equation}
i\partial_tv=-\Delta v-2\nabla X\cdot\nabla v+Yv-e^{-X}f(e^Xv)
\end{equation}
where the term $\nabla X\cdot\nabla v$ is ill-defined for $v\in L_\mu^2(\IR^d)\cap H^1(\IR^d)$. This is the main reason for which all previous work considered strong solutions with a loss of regularization, see for example the remark following Corollary $3.2$ in \cite{DW} even on $\IT^2$. With our formulation, one can then consider weak solutions with
\begin{equation}
\big\langle ie^{2X}v_t-ie^{2X}v_0,\varphi\big\rangle=\int_0^t\big\langle-(\nabla e^{2X}\nabla)v_s+e^{2X}Yv_s+e^Xf(e^Xv_s),\varphi\big\rangle\drm s
\end{equation}
for any test function $\varphi\in H^1(\IR^d)\cap L_\mu^2(\IR^d)$. For the equation on a compact domain such as $\IT^d$, this is natural as one considers initial data in the form domain of the random Anderson operator which gives solutions of the equation in its dual, the equation is tested against functions themselves in the form domain. This is not the case on full space due to the random growth of the potential at infinity thus our notation of energy solutions depends on the choice of $X$ and the initial growth $0<\mu\le 1$.

\medskip

\begin{definition}\label{def:EnergySolutions}
A function $u\in L^\infty([0,T],\CD_\mu^1)$ is called an energy solution to \eqref{GeneralEquation} if
\begin{equation}
\big\langle ie^{2X}v(t)-ie^{2X}v_0,\varphi\big\rangle=\int_0^t\big\langle-(\nabla e^{2X}\nabla)v(s)+e^{2X}Yv(s)+e^Xf(e^Xv(s)),\varphi\big\rangle\drm s
\end{equation}
for any test function $\varphi\in H^1(\IR^d)\cap L_\mu^2(\IR^d)$ and almost all $t\in[0,T]$ where $v(t)=e^{-X}u(t)$.
\end{definition}

\medskip

For any $\eps>0$, one can consider the renormalized equation
\begin{equation}
i\partial_tu_\eps=-\Delta u_\eps+u_\eps(\xi_\eps-c_\eps)+f(u_\eps)
\end{equation}
for which existence of global energy solutions follows from classical theory, for example with a spatial trunction and compactness argument. In this section, we prove existence of energy solutions using uniform bounds on $v_\eps(t)=e^{-X_\eps}u_\eps(t)$ and a compactness argument. We consider the \eqref{NLS} equation in two and three dimensions as well as the \eqref{Hartree} equation in three dimensions.

\subsection{Existence in two dimensions}

We consider the solution $u_\eps$ to
\begin{equation}
i\partial_tu_\eps=-\Delta u_\eps+u_\eps(\xi_\eps-c_\eps)+|u_\eps|^{m-1}u_\eps
\end{equation}
with initial data $u_\eps(0)=e^{X_\eps-X}u_0\in H^1\cap L_\mu^2$ where $c_\eps$ is a logarithmic diverging constant defined in Appendix \ref{Sec:stochastic_bounds}. This choice of initial data is to have regular initial condition that converges to $u_0\in\CD_\mu^1$, this is often refered to as well-prepared initial data in this context. We consider $u_\eps(t)=e^{X_\eps}v_\eps(t)$ to get
\begin{equation}
ie^{2X_\eps}\partial_tv_\eps=-(\nabla e^{2X_\eps}\nabla) v_\eps+e^{2X_\eps}Y_\eps v_\eps+e^{(m+1)X_\eps}|v_\eps|^{m-1}v_\eps
\end{equation}
with initial data $v_\eps(0)=e^{-X}u_0\in H^1\cap L_\mu^2$ and 
\begin{equation}
Y_\eps=\xi_\eps-\Delta X_\eps-|\nabla X_\eps|^2-c_\eps
\end{equation} 
which converges to $Y$ in $C_{-\delta}^{-\kappa}(\IR^2)$ for any $\kappa,\delta>0$. Using uniform bounds on $Y_\eps$ and conservation of energy, we prove the following result with $\CD_\mu^1=e^{X}\big(H^1\cap L_\mu^2\big)$.

\medskip

\begin{theorem}
% Let $0<\mu\le 1$ and $u_0\in\CD_\mu^1$. Then there exists global energy solution in $C(\IR,\CD_\mu^1)$ to the renormalized \eqref{NLS} equation.
Let $0<\mu\le1$. For any initial data $u_0\in \CD_\mu^1$, there exists energy solutions in $\CC(\IR,\CD_\mu^1)$ to \eqref{NLS} on $\IR^2$ for $m\ge1$. For any time interval $I$ containing $0$, we have the bound
\begin{equation}
\sup_{t\in I}\|u(t)\|_{\CD_\mu^1}\lesssim(1+|I|)\|u_0\|_{\CD_\mu^1}
\end{equation}
with a constant that depends only on the noise.
\end{theorem}

\medskip

\begin{proof}
The equation
\begin{equation}
i\partial_tu_\eps=-\Delta u_\eps+(\xi_\eps-c_\eps) u_\eps+|u_\eps|^{m-1}u_\eps
\end{equation}	
with $u_\eps(0)=e^{X_\eps-X}u_0\in H^1\cap L_\mu^2$ has conserved mass $\|u_\eps(t)\|_{L^2(\IR^2)}^2$ and energy
\begin{equation}
\CE_\eps(t)=\frac{1}{2}\int_{\IR^2}|\nabla u_\eps(t,x)|^2\drm x+\frac{1}{2}\int_{\IR^2}|u_\eps(t,x)|^2(\xi_\eps(x)-c_\eps)\drm x+\frac{1}{m+1}\int_{\IR^2}|u_\eps(t,x)|^{m+1}\drm x
\end{equation}
for any $\eps>0$. For $v_\eps(t)=e^{X_\eps}u_\eps(t)$ with $v_0=e^{-X}u_0\in H^1\cap L_\mu^2$, we get conserved modified mass
\begin{equation}
\int_{\IR^2}|v_\eps(t,x)|^2e^{-2X_\eps(x)}\drm x=\int_{\IR^2}|v_0|^2e^{-2X_\eps(x)}\drm x
\end{equation}
which is uniformly bounded with respect to $\eps>0$ as $e^{-2X_\eps}$ converges to $e^{-2X}$ in $L^\infty$. For the energy, we have
\begin{align}
\CE_\eps(t)=\frac{1}{2}\int_{\IR^2}&|\nabla v_\eps(t,x)|^2e^{-2X_\eps(x)}\drm x+\frac{1}{2}\int_{\IR^2}|v_\eps(t,x)|^2Y_\eps(x)e^{-2X_\eps(x)}\drm x\\
&+\frac{1}{m+1}\int_{\IR^2}|v_\eps(t,x)|^{m+1}e^{-(m+1)X_\eps(x)}\drm x
\end{align}
which is constant in time. For $v_0\in H^1\cap L_\mu^2$ with $\mu>0$, the energy $\CE_\eps(t)=\CE_\eps(0)$ is also uniformly bounded in $\eps>0$ since $Y_\eps$ converges to $Y$ in $\CC_{-\delta}^{-\kappa}(\IR^2)$ for $\delta,\kappa>0$. The equation being defocusing, we get
\begin{align}
\frac{1}{2}\int_{\IR^2}|\nabla v_\eps(t,x)|^2e^{-2X_\eps(x)}\drm x&\le\CE_\eps(0)-\frac{1}{2}\int_{\IR^2}|v_\eps(t,x)|^2Y_\eps(x)e^{-2X_\eps(x)}\drm x\\
&\lesssim\CE_\eps(0)+\|v_\eps(t)^2\|_{\CW_{2\delta}^{\kappa,1}}\|Y_\eps e^{-2X_\eps}\|_{\CC_{-2\delta}^{-\kappa}}\\
&\lesssim\CE_\eps(0)+\|v_\eps(t)\|_{H_{2\delta}^{2\kappa}}^2\\
&\lesssim\CE_\eps(0)+\|v_\eps(t)\|_{L_{\frac{2\delta}{1-2\kappa}}^2}^{2-4\kappa}\|v_\eps(t)\|_{H^1}^{4\kappa}
\end{align}
using that $Y_\eps e^{-2X_\eps}$ is uniformly bounded in $\CC_{-2\delta}^{-\kappa}$ for any $\delta,\kappa>0$ with using duality, product and interpolation of Besov spaces with Lemmas \ref{interpolation_estimate_Besov}, \ref{product_estimate_Besov} and \ref{duality_estimate_Besov}. Since $e^{2X_\eps}$ is uniformly bounded with respect to $\eps>0$ in $L^\infty$, we get
\begin{equation}
\int_{\IR^2}|\nabla v_\eps(t,x)|^2\drm x\lesssim\CE_\eps(0)+z\|v_\eps(t)\|_{L_{\frac{2\delta}{1-2\kappa}}^2}^2+z^{-\frac{1-2\kappa}{2\kappa}}c_{2\kappa}\|v_\eps(t)\|_{H^1}^2
\end{equation}
for any $z>0$ using the bound $x^{1-\theta}y^\theta\le zx+c_\theta z^{-\frac{1-\theta}{\theta}}y$ for $x,y>0$ and $c_\theta=(1-\theta)^{\frac{1-\theta}{\theta}}\theta$ with $\theta\in(0,1)$ which is a scaled version of Lemma I.3.1 from \cite{Yosida74}. It remains to bound the weighted norm. We have
\begin{align}
\frac{1}{2}\frac{\drm}{\drm t}\|v_\eps(t)e^{-X_\eps}\|_{L_{\eta}^2}^2&=\text{Re}\Big(\int_{\IR^2}\langle x\rangle^{2\eta}\partial_tv_\eps(t,x)\overline{v_\eps}(t,x)e^{-2X_\eps(x)}\drm x\Big)\\
&=\text{Im}\Big(\int_{\IR^2}\langle x\rangle^{2\eta}\big(\Delta v_\eps(t,x)-2\nabla X_\eps(x)\cdot\nabla v_\eps(t,x)\big)\overline{v_\eps}(t,x)e^{-2X_\eps(x)}\Big)\drm x\\
&=-\text{Im}\Big(\int_\IR\frac{2\eta x}{\langle x\rangle^{2\eta-2}}v_\eps(t,x)\nabla\overline v_\eps(t,x)e^{-2X_\eps(x)}\drm x\Big)\\
&\lesssim\|v_\eps(t)e^{-2X_\eps}\|_{L^2}\|\nabla v_\eps(t)\|_{L_{2\eta-1}^2}\\ 
&\lesssim\|v_0e^{-2X_\eps}\|_{L^2}\|v_\eps(t)\|_{H^1}
\end{align}
for $2\eta-1\le 0$ using the conservation of modified mass. We get
\begin{equation}
\sup_{t\in I}\|v_\eps(t)\|_{L_\eta^2}^2\lesssim\|v_0\|_{L_\eta^2}^2+|I|\cdot\|v_0\|_{L^2}\sup_{t\in I}\|v_\eps(t)\|_{H^1}
\end{equation}
for any finite interval $I$ containing $0$ with $\eta\le\mu\wedge\frac{1}{2}$ thus the energy bound with $z>0$ small enough yields
\begin{equation}
\sup_{t\in I}\int_{\IR^2}|\nabla v_\eps(t,x)|^2\drm x\lesssim\|v_0\|_{\CD_\mu^1}^2+|I|\|v_0\|_{L^2}\sup_{t\in I}\|v_\eps(t)\|_{H^1}.
\end{equation}
Indeed, one can put some weight on the gradient to get up to $0<\mu\le1$ using iterated Gronwall's lemma as in one dimensions. We get
\begin{equation}
\sup_{\eps>0}\ \sup_{t\in I}\|v_\eps(t)\|_{H^1}\lesssim(1+|I|)\|v_0\|_{\CD_\mu^1}
\end{equation}
for any finite interval $I\subset\IR$ which is uniformly bounded with respect to $\eps>0$. Using classical compactness argument as explained in the proof of Theorem \ref{thm:WP1D}, we get that there exists $v\in\CC(\IR,\CD_\mu^1)$ such that $v_\eps$ converges weakly in $H^1\cap L_\mu^2$ at almost all time $t\in\IR$ and to $v$ in $\CC(\IR,H_\delta^\sigma)$ for any $\delta<\mu$ and $\sigma<1$ up to extraction. Using this convergence in the equation
\begin{equation}
ie^{2X_\eps}\partial_tv_\eps=-\nabla e^{2X_\eps}\nabla v_\eps+e^{2X_\eps}Y_\eps v_\eps+e^{(m+1)X_\eps}|v_\eps|^{m-1}v_\eps
\end{equation}
gives that $v$ sastifies
\begin{equation}
ie^{2X}\partial_tv=-\nabla e^{2X}\nabla v+e^{2X}Y v+e^{(m+1)X}|v|^{m-1}v
\end{equation}
using also that $X_\eps$ converges to $X$ and $Y_\eps$ to $Y$ respectively in $\CC_{\delta}^{1-\kappa}$ and $\CC_{-\delta}^{-\kappa}$ hence $v$ is an energy solution in the sense of Definition \ref{def:EnergySolutions}. For example, the nonlinear term is controlled in $H_\delta^{-1}$ using the embeddings
\begin{equation}
H_\delta^\sigma(\IR^2)\hookrightarrow L_\delta^{m}(\IR^2)\hookrightarrow H_\delta^{-1}
\end{equation}
for $\sigma\ge 1-\frac{2}{m}$ and $e^{(m+1)X}\in L^\infty(\IR^2)$.
\end{proof}

\begin{remark}
In the previous proof, we consider the regularized equation
\begin{equation}
ie^{2X_\eps}\partial_tv_\eps=-\nabla e^{2X_\eps}\nabla v_\eps+e^{2X_\eps}Y_\eps v_\eps+e^{(m+1)X_\eps}|v_\eps|^{m-1}v_\eps
\end{equation}
where both sides converges in $H^{-1}(\IR^2)$. One can also consider the equation
\begin{equation}
i\partial_tv_\eps=-e^{-2X_\eps}\nabla e^{2X_\eps}\nabla v_\eps+Y_\eps v_\eps+e^{(m-1)X_\eps}|v_\eps|^{m-1}v_\eps
\end{equation}
however this does not converges a priori since $e^{2X_\eps}$ converges only in $\CC_\delta^{1-\kappa}(\IR^2)$. In two dimensions, the sum of the regularity is $-2\kappa<0$ and one can consider this equation as soon as one propagate any regularity better than $H_{\text{loc}}^1(\IR^2)$. This is the main reason that motivated Debussche and Weber to consider strong solutions in \cite{DW} and all previous works.
\end{remark}

\subsection{Existence in three dimensions}

With our notation, the computations are similar in three dimensions with different regularity exponent. The difficulty here lies in the construction of $X$ and $Y$ such as
\begin{equation}
Y=\lim_{\eps\to0}(\xi_\eps-\Delta X_\eps-|\nabla X_\eps|^2-c_\eps)
\end{equation}
in $\CC_{-\delta}^{-\frac{1}{2}-\kappa}(\IR^3)$ while $X_\eps$ converges to $X$ in $\CC_{\delta}^{\frac{1}{2}-\kappa}$ for any $\delta,\kappa>0$, this is done in Appendix \ref{Sec:stochastic_bounds}. We also consider the solution $u_\eps$ to
\begin{equation}
i\partial_tu_\eps=-\Delta u_\eps+u_\eps(\xi_\eps-c_\eps)+|u_\eps|^{m-1}u_\eps
\end{equation}
with well-prepared initial data $u_\eps(0)=e^{X_\eps-X}u_0\in H^1\cap L_\mu^2$. For $u_\eps=e^{X_\eps}v_\eps$, we get
\begin{equation}
ie^{2X_\eps}\partial_tv_\eps=-\nabla e^{2X_\eps}\nabla v_\eps+e^{2X_\eps}Y_\eps v_\eps+e^{(m+1)X_\eps}|v_\eps|^{m-1}v_\eps
\end{equation}
with initial data $v_\eps(0)=v_0=e^{-X}u_0\in H^1\cap L_\mu^2$. Using uniform bounds on $Y_\eps$ in $\CC_{-\delta}^{-\frac{1}{2}-\kappa}(\IR^2)$ and conservation of energy, we prove the following result with $\CD_\mu^1=e^{X}\big(H^1\cap L_\mu^2\big)$ with a condition on $m\ge1$ due to the regularity of $Y$ and the Besov embedding in three dimensions.

\medskip

\begin{theorem}
Let $0<\mu\le1$. For any initial data $u_0\in \CD_\mu^1$, there exists energy solutions in $\CC(\IR,\CD_\mu^1)$ to \eqref{NLS} on $\IR^3$ for $1\le m<6$. For any time interval $I$ containing $0$, we have the bound
\begin{equation}
\sup_{t\in I}\|u(t)\|_{\CD_\mu^1}\lesssim(1+|I|)\|u_0\|_{\CD_\mu^1}
\end{equation}
with a constant that depends only on the noise.
\end{theorem}

\medskip

\begin{proof}
We follow the proof for the same result on $\IR^2$ and emphasize the difference in the regularity exponents. We have the same convervation of modified mass
\begin{equation}
\int_{\IR^3}|v_\eps(t,x)|^2e^{-2X_\eps(x)}\drm x=\int_{\IR^3}|v_0|^2e^{-2X_\eps(x)}\drm x
\end{equation}
and modified energy
\begin{align}
\CE_\eps(t)=\frac{1}{2}\int_{\IR^3}&|\nabla v_\eps(t,x)|^2e^{-2X_\eps(x)}\drm x+\frac{1}{2}\int_{\IR^3}|v_\eps(t,x)|^2Y_\eps(x)e^{-2X_\eps(x)}\drm x\\
&+\frac{1}{m+1}\int_{\IR^3}|v_\eps(t,x)|^{m+1}e^{-(m+1)X_\eps(x)}\drm x
\end{align}
which are finite at $t=0$ since $v_0\in H^1\cap L_\mu^2$ with $\mu>0$. In particular, $Y_\eps$ is uniformly bounded in $\CC_{-\delta}^{-\frac{1}{2}-\kappa}$ hence $v\in H^1$ is enough for all the term to be well-defined. We get
\begin{align}
\frac{1}{2}\int_{\IR^3}|\nabla v_\eps(t,x)|^2e^{-2X_\eps(x)}\drm x&\le\CE_\eps(0)-\frac{1}{2}\int_{\IR^3}|v_\eps(t,x)|^2Y_\eps(x)e^{-2X_\eps(x)}\drm x\\
&\lesssim\CE_\eps(0)+\|v_\eps(t)^2\|_{\CW_{2\delta}^{\frac{1}{2}+\kappa,1}}\|e^{-2X_\eps}Y_\eps\|_{\CC_{-2\delta}^{-\frac{1}{2}-\kappa}}\\
&\lesssim\CE_\eps(0)+\|v_\eps(t)\|_{H_{2\delta}^{\frac{1}{2}+2\kappa}}^2\\
&\lesssim\CE_\eps(0)+\|v_\eps(t)\|_{L_{\frac{4\delta}{1-4\kappa}}^2}^{1-4\kappa}\|v_\eps(t)\|_{H^1}^{1+4\kappa}
\end{align}
using that $e^{-2X_\eps}Y_\eps$ is uniformly bounded in $\CC_{-2\delta}^{-\frac{1}{2}-\kappa}$ for any $\delta,\kappa>0$ with Lemmas \ref{interpolation_estimate_Besov}, \ref{product_estimate_Besov} and \ref{duality_estimate_Besov}. Since $e^{2X_\eps}$ is uniformly bounded with respect to $\eps>0$ in $L^\infty$, we get
\begin{equation}
\int_{\IR^3}|\nabla v_\eps(t,x)|^2\drm x\lesssim\CE_\eps(0)+z\|v_\eps(t)\|_{L_{\frac{4\delta}{1-4\kappa}}^2}^2+z^{-\frac{1-4\kappa}{1+4\kappa}}\|v_\eps(t)\|_{H^1}^2
\end{equation}
for any $z>0$ as in two dimensions and it remains to bound the weighted norm. We have
\begin{equation}
\frac{\drm}{\drm t}\|v_\eps(t)e^{-X_\eps}\|_{L_{\eta}^2}^2\lesssim\|v_0e^{-2X_\eps(x)}\|_{L^2}\|v_\eps(t)\|_{H^1}
\end{equation}
for $2\eta-1\le 0$ using the conservation of modified mass hence
\begin{equation}
\sup_{t\in I}\|v_\eps(t)\|_{L_\eta^2}^2\lesssim\|v_0\|_{L_\eta^2}^2+|I|\sup_{t\in I}\|v_\eps(t)\|_{H^1}
\end{equation}
for any finite interval $I$ containing $0$ with $\eta\le\mu\wedge\frac{1}{2}$. Again, the energy bound then yields
\begin{equation}
\sup_{t\in I}\int_{\IR^3}|\nabla v_\eps(t,x)|^2\drm x\lesssim\|v_0\|_{\CD_\mu^1}+|I|\sup_{t\in I}\|v_\eps(t)\|_{H^1}^{\frac{3}{2}+2\kappa}
\end{equation}
which gives the bound
\begin{equation}
\sup_{\eps>0}\ \sup_{t\in I}\|v_\eps(t)\|_{H^1}\lesssim(1+|I|)\|v_0\|_{\CD_\mu^1}
\end{equation}
for any finite interval $I\subset\IR$ uniform with respect to $\eps>0$. The compactness argument gives a solution $v\in C(\IR,\CD_\mu^1)$ such that $v_\eps$ converges to $v$ in $L^\infty(\IR,H_\delta^\sigma)$ for any $\delta<\mu$ and $\sigma<1$ up to extraction. It is a solution to
\begin{equation}
ie^{2X}\partial_tv=-\nabla e^{2X}\nabla v+e^{2X}Y v+e^{(m+1)X}|v|^{m-1}v
\end{equation}
using also that $X_\eps$ converges to $X$ and $Y_\eps$ to $Y$ respectively in $\CC_{\delta}^{\frac{1}{2}-\kappa}$ and $\CC_{-\frac{1}{2}-\delta}^{-\kappa}$ hence $v$ is an energy solution in the sense of Definition \ref{def:EnergySolutions}. The main difference here is the control of the nonlinear term which imposes a restriction on $m\ge1$. In three dimension, we have the embeddings
\begin{equation}
H_\delta^\sigma(\IR^3)\hookrightarrow L_\delta^p(\IR^3)
\end{equation}
for $\sigma\ge \frac{3}{2}-\frac{3}{p}$, that is $p\le\frac{6}{3-2\sigma}$.	Since $\sigma<1$, we can control $L_\delta^p(\IR^3)$ with $p<6$ hence the restriction. 
\end{proof}

We now prove a similar theorem for \eqref{Hartree} equation.

\medskip

\begin{theorem}
Let $0<\mu\le1$. For any initial data $u_0\in \CD_\mu^1$, there exists energy solutions in $\CC(\IR,\CD_\mu^1)$ to \eqref{Hartree} on $\IR^3$ for $\beta>0$.
\end{theorem}

\medskip

\begin{proof}
The difference with the previous proof is the control of the nonlinearity
\begin{equation}
u_\eps\cdot V_\beta*|u_\eps|^2
\end{equation}
for $\eps>0$. The energy is
\begin{equation}
\CE_\eps(t)=\frac{1}{2}\int_{\IR^3}|\nabla u_\eps(t,x)|^2\drm x+\frac{1}{2}\int_{\IR^3}|u_\eps(t,x)|^2(\xi_\eps(x)-c_\eps)\drm x+\frac{1}{4}\int_{\IR^3}|u_\eps(t,x)|^2(V_\beta*|u_\eps(t)|^2)(x)\drm x
\end{equation}
where $V$ is symmetric thus the only difference in the modified energy is the term
\begin{equation}
\int_{\IR^3}|v_\eps(t,x)|^2(V_\beta*|e^{-2X_\eps}v_\eps(t)|^2)(x)e^{-2X_\eps(x)}\drm x
\end{equation}
which is nonnegative since $V\ge0$. Following the previous proof, we only need to bound this at $t=0$ using $v_0\in H^1(\IR^3)\hookrightarrow L^6(\IR^3)$. Using $e^{-2X_\eps}\in L^\infty$ is uniformly bounded, we have
\begin{align}
\|V_\beta*|e^{-2X_\eps}v_0|^2\|_{L^3(\IR^3)}&\lesssim\|V_\beta\|_{L^1(\IR^3)}\|e^{-2X_\eps}v_0|^2\|_{L^3(\IR^3)}\\
&\lesssim\|V_\beta\|_{W^{1,\beta}}\|v_0\|_{L^6(\IR^d)}
\end{align}
using Young inequality for convolution and $W^{1,\beta}(\IR^3)\hookrightarrow L^1(\IR^3)$ for any $\beta>0$. We get
\begin{align}
\int_{\IR^3}|v_0|^2(V_\beta*|e^{-2X_\eps}v_0|^2)(x)e^{-2X_\eps(x)}\drm x&\lesssim\|v_0\|_{L^{\frac{3}{2}}}^2\|V_\beta*|e^{-2X_\eps}v_0|^2\|_{L^3(\IR^3)}\\
&\lesssim\|v_0\|_{H^1}^3\|V_\beta\|_{W^{1,\beta}}
\end{align}
which completes the proof.
\end{proof}

\section{Strong solutions and Strichartz inequalities}\label{Sec:StrongandStrichartz}

This section deals with the linear equation, propagation of regularity for strong solutions is obtained which is then used in our proof of Strichartz inequalities.

\subsection{Propagation of regularity}

The previous construction also gives existence of solutions to the linear equation
\begin{equation}
i\partial_tu=\CH u
\end{equation}
interpreted as
\begin{equation}\label{LinearEquation}
ie^{2X}\partial_tv=e^X\CH e^Xv
\end{equation}
with the new variable $u=e^Xv$ for $v_0\in L_\mu^2(\IR^d)\cap H^1(\IR^d)$ for $0<\mu\le 1$. For the linear equation, one can directly obtain uniqueness which gives a construction of the associated linear flow.

\medskip

\begin{proposition}
Let $0<\mu\le1$. There exists a unique solution in $\CC(\IR,L_\mu^2(\IR^d)\cap H^1(\IR^d))$ to the linear equation \eqref{LinearEquation} for $v_0\in L_\mu^2(\IR^d)\cap H^1(\IR^d)$. Denoting as $S_t$ the associated flow, it satisfies
\begin{equation}
\|S_tv_0\|_{L_{\sigma}^2(\IR^d)\cap H^\sigma(\IR^d)}\lesssim\|v_0\|_{L_{\sigma}^2(\IR^d)\cap H^\sigma(\IR^d)}
\end{equation}
for any $\sigma\in[0,1]$. 
\end{proposition}

\medskip

\begin{proof}
The existence follows from the same bounds as for the nonlinear equation for $v_0\in L_\mu^2(\IR^d)\cap H^1(\IR^d)$. For uniqueness, it is enough to prove that the solution is $0$ for $v_0=0$ by linearity. It follows from the computation
\begin{align}
\frac{1}{2}\frac{\drm}{\drm t}\int_{\IR^d}|v(t,x)|^2e^{2X(x)}\drm x&=\text{Re}\int_{\IR^d}\partial_tv(t,x)\overline{v(t,x)}e^{2X(x)}\drm x\\
&=-\text{Re}\int_{\IR^d}ie^X(\CH e^Xv)(t,x)\overline{v(t,x)}\drm x\\
&=\text{Re}\int_{\IR^d}\big(i\nabla e^{2X}\nabla v)(t,x)+Y(x)v(t,x)\big)\overline{v(t,x)}\drm x\\
&=0
\end{align}
which is nothing more that the conservation of the modified mass. Since there exists a unique solution for any $v_0\in L_\mu^2(\IR^d)\cap H^1(\IR^d)$ for any horizon time $T>0$, we denote as $S_t$ the associated flow. The conservation of modified mass implies that the flow satisfies
\begin{equation}
\|S_tv_0\|_{L^2(\IR^d,e^{2X(x)}\drm x)}=\|v_0\|_{L^2(\IR^d,e^{2X(x)}\drm x)}
\end{equation}
for any $v_0\in L_\mu^2(\IR^d)\cap H^1(\IR^d)$ which is a dense subspace of $L^2(\IR^d)$. Since the usual norm on $L^2(\IR^d)$ is equivalent to the modified mass, $S_t$ extends uniquely as a continuous operator on $L^2(\IR^d)$. Considering $\mu=1$, we also proved the bound
\begin{equation}
\|S_tv_0\|_{L_1^2(\IR^d)\cap H^1(\IR^d)}\lesssim\|v_0\|_{L_1^2(\IR^d)\cap H^1(\IR^d)}
\end{equation}
hence interpolation, for example see \cite{BonjioanniTorrea06} Lemma 3, gives
\begin{equation}
\|S_tv_0\|_{L_{\sigma}^2(\IR^d)\cap H^\sigma(\IR^d)}\lesssim\|v_0\|_{L_{\sigma}^2\cap H^\sigma(\IR^d)}
\end{equation}
for any $\sigma\in[0,1]$.
\end{proof}

\begin{remark}\label{Remark:linearFlow}
\textbf{(i)} We have a bound where the exponent weight has to be precisely the regularity exponent, this comes from the interpolation between $L^2(\IR^d)$ and $L_\mu^2(\IR^d)\cap H^1(\IR^d)$ which is the energy space of the harmonic oscillator. We did not find a precise interpolation between $L^2(\IR^d)$ and $L_1^2(\IR^d)\cap H^1(\IR^d)$ for general $\mu\in(0,1)$ while we believe that it follows from the same line of reasoning. It would give the more general bound
\begin{equation}
\|S_tv_0\|_{L_{\sigma\mu}^2(\IR^d)\cap H^\sigma(\IR^d)}\lesssim\|v_0\|_{L_{\sigma\mu}^2(\IR^d)\cap H^\sigma(\IR^d)}
\end{equation}
for $0<\mu\le \sigma\le1$ hence remove the condition $\mu\ge\gamma$ in Theorem \ref{thm:Strichartz}.

\smallskip

\textbf{(ii)} The flow is also continuous from $L^2(\IR^d)$ to itself by extension of a continuous linear operator. This naturally gives a definition of the Anderson Hamiltonian on $\IR^d$ for $d\le3$ as the generator of the strongly continuous one-parameter unitary group $(S_t)_{t\in\IR}$. Hsu and Labbé \cite{HsuLabbe25} recently used the heat semigroup to construct this operator. In their case, the domain of the semigroup depends on $t>0$ while this is not the case for the Schrödinger group defined on $L^2(\IR^d)$. See also \cite{Ueki25} for a direct construction of the operator on $\IR^2$ using paracontrolled calculus.
\end{remark}

We now prove propagation of regularity beyond energy solution for the linear equation in our weighted framework with the space 
\begin{equation}
L_\mu^2(\IR^d)\cap\Gamma H^2(\IR^d)
\end{equation}
for any $0<\mu\le 1$ where $\Gamma$ is defined using paracontrolled calculus in Appendix \ref{sec:construction-AH}. It is implicitly defined by the paracontrolled expansion
\begin{equation}
\Gamma u^\sharp=\P_{\nabla\Gamma u^\sharp}\nabla Z_1+\P_{\Gamma u^\sharp}Z_2+u^\sharp
\end{equation}
for well chosen random field $Z_1$ and $Z_2$ for $u^\sharp\in H^2(\IR^d)$, the construction being harder in three dimensions. The map $\Gamma$ can be extended to $L^2(\IR^d)$ and we have the following Proposition.

\medskip

\begin{proposition}
For $\mu\in(0,1]$, the flow $S_t$ is continuous from $L_\mu^2(\IR^d)\cap\Gamma H^2(\IR^d)$ to itself. In particular, we have
\begin{equation}
\|S_tv_0\|_{L_{\mu}^2(\IR^d)\cap\Gamma H^\sigma(\IR^d)}\lesssim\|v_0\|_{L_{\mu}^2(\IR^d)\cap \Gamma H^\sigma(\IR^d)}
\end{equation}
for any $\sigma\in[1,2]$.
\end{proposition}

\medskip

\begin{proof}
For $v_0\in L_\mu^2\cap \Gamma H^2\subset L_\mu^2\cap H^1$, there exists a unique solution $v$ to the linear equation 
\begin{equation}
ie^{2X}\partial_tv=e^X\CH e^Xv
\end{equation}
with $v(0)=v_0$. Then $w=\partial_tv$ satisfies the same linear equation with
\begin{equation}
w(0)=e^X\CH e^Xv_0\in L_{\mu'}^2(\IR^d)
\end{equation}
for any $\mu'<\mu$ using that $v_0\in L_\mu^2\cap\Gamma H^2$. Since $w_t=S_tw(0)$, we get
\begin{equation}
\|w_t\|_{L_{\mu'}^2(\IR^d)}\lesssim\|v_0\|_{L_\mu^2\cap \Gamma H^2}
\end{equation}
hence
\begin{equation}
\|e^X\CH e^Xv_t\|_{L_{\mu'}^2(\IR^d)}\lesssim\|v_0\|_{L_\mu^2\cap \Gamma H^2}
\end{equation}
using the equation on $v$. This gives
\begin{equation}
\|v_t\|_{L_{\mu'}^2\cap \Gamma H^2}\lesssim\|v_0\|_{L_\mu^2\cap \Gamma H^2}
\end{equation}
using Corollary \ref{cor:EquivalenceDomainSpaces} thus the flow
\begin{equation}
S_t^\sharp:=\Gamma^{-1}S_t\Gamma
\end{equation}
is continuous from $L_\mu^2(\IR^d)\cap H^\sigma(\IR^d)$ to itself for $\sigma\in\{1,2\}$ and interpolation completes the proof.
\end{proof}

\begin{remark}\label{remark:vsharpEquation}
One can prove that $\Gamma H^\sigma(\IR^d)=H^\sigma(\IR^d)$ for $\sigma<2$ in two dimensions and $\sigma<\frac{3}{2}$ in three dimensions hence propagation of regularity beyond energy solutions for the equation on $v$. For $d=2$, we get $v_t\in H^{1+\gamma}(\IR^2)$ for $v_0\in H^{1+\gamma}$ and any $\gamma\in(0,1)$. We get that $e^X\CH e^Xv\in H^{-1+\gamma}(\IR^2)$ for $v\in H^{1+\gamma}(\IR^2)$ hence its products with $e^{-2X}\in\CC^{1-\kappa}(\IR^2)$ is well-defined. This gives the equation 
\begin{equation}
i\partial_tv=e^{-X}\CH e^Xv
\end{equation}
which involves the gradient term $\nabla X\cdot\nabla v$ since $\gamma>0$. For the linear equation, this recovers the notion of solution of Debussche and Weber \cite{DW} on the torus for example. In three dimensions, we have propagation of $H^\sigma(\IR^3)$ for $\sigma<\frac{3}{2}$ while $e^{-2X}\in\CC^{\frac{1}{2}-\kappa}$ hence this does not work. Using the paracontrolled calculus, one can consider the conjugated flow
\begin{equation}
S_t^\sharp=\Gamma^{-1}S_t\Gamma
\end{equation}
which is then continuous from $L_\mu^2(\IR^d)\cap H^\sigma(\IR^d)$ to itself for any $\sigma\in[1,2]$ and $0<\mu\le1$ and we get the equation
\begin{equation}
i\partial_tv^\sharp=\Gamma^{-1}e^{-X}\CH e^X\Gamma v^\sharp
\end{equation}
with $v=\Gamma v^\sharp$ for $\sigma\in(\frac{3}{2},2)$. In particular, $S_t^\sharp$ is the flow associated to $\CH^\sharp=(\Gamma e^X)^{-1}\CH (e^X\Gamma)$.
\end{remark}

\subsection{Strichartz inequalities}

Let $u_\eps$ and $u_\eps'$ be two solutions to \eqref{GeneralEquation} with a regularized renormalized noise and consider $w_\eps=u_\eps-u_\eps'$. We get
\begin{equation}
i\partial_tw_\eps=\CH_\eps w_\eps+f(u_\eps)-f(u_\eps')
\end{equation}
with $f$ the polynomial or Hartree nonlinearity. Multiply the equation by $\overline{w_\eps}$, integrate in time and take the imaginary part gives
\begin{equation}
\frac{\drm }{\drm t}\int_{\IR^d}|w_\eps(t,x)|^2\drm x=\text{Im}\Big(\int_{\IR^d}\big(f(u_\eps)-f(u_\eps')\big)(t,x)\overline{w_\eps(t,x)}\drm x\Big)
\end{equation}
where one has to specify the nonlinearity at this point. For example in the cubic case, we have
\begin{equation}
f(u_\eps)-f(u_\eps')=|u_\eps|^2u_\eps-|u_\eps'|^2u_\eps'
\end{equation} 
which gives
\begin{equation}
\frac{\drm }{\drm t}\int_{\IR^d}|w_\eps(t,x)|^2\drm x\lesssim\int_{\IR^d}(|u_\eps(t,x)|^2+|u_\eps'(t,x)|^2)|w_\eps(t,x)|^2\drm x.%\int_{\IR^d}|f(u_\eps)-f(u_\eps')|^2(t,x)\drm x\int_{\IR^d}|w_\eps(t,x)|^2\drm x.
\end{equation}
hence with Grönwall's lemma, we get
\begin{equation}
\int_{\IR^d}|w_\eps(T,x)|^2\drm x\lesssim e^{c\int_0^T(\|u_\eps(t)\|_{L^\infty}^2+\|u_\eps'(t)\|_{L^\infty}^2)\drm t}\int_{\IR^d}|u_0(x)-u_0'(x)|^2\drm x
\end{equation}
for any $T>0$. From this estimate, uniqueness follows from a bound on $u$ and $u'$ in $L^2([0,T],L^\infty(\IR^d))$ for the cubic nonlinearity. While conservation of mass yields solutions in $L^\infty([0,T],L^2(\IR^d))$, the goal is to trade time integrability for spatial integrability. In particular, the initial data does not belong to $L^\infty(\IR^d)$ for energy solutions with $d\in\{2,3\}$ and this requires dispersive properties of the solution. This is exactly encoded in Strichartz inequalities with the following result from Keel and Tao \cite{KeelTao1998} in the deterministic case.

\medskip

\begin{theorem}\label{Thm:KeelTao}
Let $T>0$ and $\mu\in(0,1)$. Let $(p_1,q_2)$ and $(p_2,q_2)$ two Strichartz pairs, that is
\begin{equation}
\frac{2}{p_i}+\frac{d}{q_i}=\frac{d}{2}
\end{equation}
for $i\in\{1,2\}$ with $(p_i,q_i)\neq(2,\infty)$ in two dimensions. Let $w$ be a solution to
\begin{equation}
i\partial_tw=-\Delta w+F
\end{equation}
with initial data $w_0$. Then the solution satisfies the bound
\begin{equation}
\|w\|_{L^p([0,T],L^q(\IR^d))}\lesssim\|w_0\|_{L^2(\IR^d)}+\|F\|_{L^{p_2'}([0,T],L^{q_2'}(\IR^d))}
\end{equation}
where $p_2'$ and $q_2'$ denotes the conjugated exponent of $p_2$ and $q_2$.
\end{theorem}

\medskip

In this section, we obtain Strichartz inequalities for \eqref{GeneralEquation} in the linear case $f=0$ using this result for the deterministic equation. The proof relies on the transformed operator
\begin{equation}
\CH^\sharp=(e^X\Gamma)^{-1}\CH(e^X\Gamma)
\end{equation}
introduced in Appendix \ref{sec:construction-AH} with Proposition \ref{prop:fineHcomparison} which gives
\begin{equation}
\|(\CH^\sharp+\Delta)v^\sharp\|_{H^{\gamma-\kappa}(\IR^d)}\lesssim\|v^\sharp\|_{H^{\frac{d}{2}+\gamma}(\IR^d)\cap L_\mu^2(\IR^d)}
\end{equation}
for any $\kappa>0$ and $\gamma\in(0,2-\frac{d}{2})$. The construction of the $\Gamma$ map requires paracontrolled calculus which explains that $\CH^\sharp+\Delta$ takes values in a space of positive regularity, this is in contrast with the weak bound
\begin{equation}
\|(e^{-X}\CH e^X+\Delta)v\|_{H^{-\kappa}(\IR^d)}\lesssim\|v\|_{H^{\frac{d}{2}+\gamma}(\IR^d)\cap L_\mu^2(\IR^d)}
\end{equation}
for $\gamma>0$ where one does not gain from the regularity of $v$. We prove the following Strichartz inequalities following \cite{MZ,DeVecchiJiZachhuber25}. Paracontrolled calculus is only needed for Proposition \ref{prop:fineHcomparison} which is proved in the Appendix and can be considered as a black box. In the previous section, we proved continuity properties of the flow $S_t$ associated to the linear equation
\begin{equation}
ie^{2X}\partial_tv=e^X\CH e^Xv
\end{equation}
for $v_0\in L_\mu^2(\IR^2)\cap H^1(\IR^2)$. We naturally have the conjugated flow
\begin{equation}
S_t^\sharp=\Gamma^{-1}S_t\Gamma
\end{equation}
which is continuous from $L_\mu^2(\IR^d)\cap H^\sigma(\IR^d)$ to itself for any $\sigma\in[1,2]$. As explained in Remark \ref{remark:vsharpEquation}, considering $v=\Gamma v^\sharp$ with $v^\sharp\in H^\sigma$ for $\sigma\in(\frac{d}{2},2)$ gives the equation
\begin{equation}
i\partial_tv^\sharp=\CH^\sharp v^\sharp
\end{equation}
with $v^\sharp(0)=\Gamma^{-1}v_0$ since $\Gamma$ does not depend on time. One can then consider the mild formulation associated to the Laplacian
\begin{equation}
v_t^\sharp=e^{it\Delta}v_0^\sharp-i\int_0^te^{i(t-s)\Delta}(\CH^\sharp+\Delta)v_s^\sharp\drm s
\end{equation}
using the classical Schrödinger group. As done in \cite{MZ} on compact surfaces, we use this decomposition to obtain the following Strichartz inequalities.

\medskip

\begin{theorem}\label{thm:Strichartz}
Let $(p,q)$ be a Strichartz pair. We have
\begin{equation}
\|S_t^\sharp v_0\|_{L^p([0,1],W^{\gamma,q}(\IR^d))}\lesssim\|v_0\|_{L_\mu^2\cap H^{\frac{d}{2p}+\kappa+\gamma}}
\end{equation}
for any $\gamma\in[0,2-\frac{d}{2p})$ and $\mu\in(\gamma,1)$ with $\kappa>0$ small enough. For the initial flow, we get
\begin{equation}
\|S_t v_0\|_{L^p([0,1],W^{\gamma,q}(\IR^d))}\lesssim\|v_0\|_{L_\mu^2\cap H^{\frac{d}{2p}+\kappa+\gamma}}
\end{equation}
for any $\frac{d}{2p}+\gamma<3-\frac{d}{2}$ and $\mu\in(\gamma,1)$.
\end{theorem}

\medskip

\begin{proof}
We work with the Paley-Littlewood projectors $\Delta_j$ on an annulus of size $2^j$ in frequencies. Let $k,j\ge0$ be fixed. For any $N\ge1$, we have
\begin{equation}
\|\Delta_jS_t^\sharp\Delta_kv\|_{L^p([0,1],W^{\gamma,q}(\IR^d))}^p=\sum_{n=0}^{N-1}\|\Delta_jS_t^\sharp\Delta_kv\|_{L^p([t_n,t_{n+1}],W^{\gamma,q}(\IR^d))}^p
\end{equation}
with $t_n=\frac{n}{N}$ for $0\le n\le N$. For $t\in[t_n,t_{n+1}]$, we have
\begin{align*}
S_t^\sharp\Delta_kv&=S_{t-t_n}^\sharp S_{t_n}^\sharp\Delta_kv\\
&=e^{-i(t-t_n)\Delta}S_{t_n}^\sharp\Delta_kv+\int_{t_n}^te^{-i(t-s)\Delta}\big(\CH^\sharp+\Delta\big)S_s^\sharp\Delta_kv\drm s
\end{align*}
using the mild formulation for $S_{t-t_n}^\sharp$ with respect to the Laplacian hence
\begin{align*}
\|\Delta_jS_t^\sharp\Delta_kv\|_{L^p([0,1],W^{\gamma,q}(\IR^d))}^p\le\sum_{n=0}^{N-1}&\|\Delta_je^{-i(t-t_n)\Delta}S_{t_n}^\sharp\Delta_kv\|_{L^p([t_n,t_{n+1}],W^{\gamma,q}(\IR^d))}^p\\
&+\|\int_{t_n}^t\Delta_je^{-i(t-s)\Delta}\big(\CH^\sharp+\Delta\big)S_s^\sharp\Delta_kv\drm s\|_{L^p([t_n,t_{n+1}],W^{\gamma,q}(\IR^d))}^p.
\end{align*}
A crucial point is that the projector $\Delta_j$ commutes with $e^{it\Delta}$ while this is not true for $S_t^\sharp$. For the first term, we have
\begin{align*}
\|\Delta_je^{-i(t-t_n)\Delta}S_{t_n}^\sharp\Delta_kv\|_{L^p([t_n,t_{n+1}],W^{\gamma,q}(\IR^d))}&=\|e^{-i(t-t_n)\Delta}\Delta_jS_{t_n}^\sharp\Delta_kv\|_{L^p([t_n,t_{n+1}],W^{\gamma,q}(\IR^d))}\\
&\lesssim \|\Delta_jS_{t_n}^\sharp\Delta_kv\|_{H^\gamma}\\
&\lesssim2^{-j\delta}\|S_{t_n}^\sharp\Delta_kv\|_{H^{\gamma+\delta}}\\
&\lesssim2^{-j\delta}\|\Delta_kv\|_{L_\mu^2\cap H^{\gamma+\delta}}\\
&\lesssim2^{-j\delta}2^{-k\delta'}\|\Delta_kv\|_{L_\mu^2\cap H^{\gamma+\delta+\delta'}}
\end{align*}
for any $\delta,\delta',\eps>0$ using Bernstein lemma and continuity of $S_t^\sharp$ on $L_\mu^2\cap H^{\gamma+\delta}$ to $H^{\gamma+\delta}$ for $\gamma+\delta\le 2$ and $\mu\ge\gamma+\delta$ if $\gamma+\delta<1$. For the second term, we have
\begin{align*}
\|\int_{t_n}^t\Delta_je^{-i(t-s)\Delta}\big(\CH^\sharp+\Delta\big)S_s^\sharp\Delta_kv\drm s\|_{L^p([t_n,t_{n+1}],W^{\gamma,q}(\IR^d))}&\le\int_{t_n}^{t_{n+1}}\|\Delta_j\big(\CH^\sharp+\Delta\big)S_s^\sharp\Delta_kv\|_{H^\gamma}\drm s\\
&\le\int_{t_n}^{t_{n+1}}2^{-j\sigma}\|\big(\CH^\sharp+\Delta\big)S_s^\sharp\Delta_kv\|_{H^{\gamma+\sigma}}\drm s\\
&\le\int_{t_n}^{t_{n+1}}2^{-j\sigma}\|S_s^\sharp\Delta_kv\|_{H^{\gamma+\sigma+\frac{d}{2}-\kappa}}\drm s\\
&\lesssim N^{-1}2^{-j\sigma}\|\Delta_kv\|_{L_\mu^2\cap H^{\gamma+\sigma+\frac{d}{2}-\kappa}}\\
&\lesssim N^{-1}2^{-j\sigma}2^{-k\sigma'}\|\Delta_kv\|_{L_\mu^2\cap H^{\gamma+\sigma+\sigma'+\frac{d}{2}-\kappa}}
\end{align*}
for any $\sigma,\sigma',\eps>0$ with the same arguments in addition to the control of $\CH^\sharp+\Delta$ in a positive Sobolev space for $\gamma+\sigma<2-\frac{d}{2}$. We get
\begin{equation}
\|\Delta_jS_t^\sharp\Delta_kv\|_{L^p([0,1],W^{\gamma,q}(\IR^d))}\lesssim N^{\frac{1}{p}}2^{-j\delta}2^{-k\delta'}\|\Delta_kv\|_{L_\mu^2\cap H^{\gamma+\delta+\delta'}}+N^{-\frac{p-1}{p}}2^{-j\sigma}2^{-k\sigma'}\|\Delta_kv\|_{L_\mu^2\cap H^{\gamma+\sigma+\sigma'+\frac{d}{2}-\kappa}}
\end{equation}
and we choose different parameters to deal with the sums $k\le j$ and $k>j$. For the first sum, consider
\begin{equation}
\delta=\frac{1}{p}\eta'+\kappa,\delta'>0,\sigma>0,\sigma'>0,N=2^{\eta'j},\eta=\frac{p-1}{p}\eta',\eta'=\frac{d}{2}-\kappa
\end{equation}
which gives
\begin{align*}
\sum_{k\le j}\|\Delta_j&S_t^\sharp\Delta_kv\|_{L^p([0,1],W^{\gamma,q}(\IR^d))}\lesssim\sum_{j\ge0}\sum_{k\le j}N^{\frac{1}{p}}2^{-j\delta}2^{-k\delta'}\|\Delta_kv\|_{L_\mu^2\cap H^{\gamma+\delta+\delta'}}\\
&\hspace{4cm}+N^{-\frac{p-1}{p}}2^{-j\sigma}2^{-k\sigma'}\|\Delta_kv\|_{L_\mu^2\cap H^{\gamma+\sigma+\sigma'+\frac{d}{2}-\kappa}}\\
&\lesssim\sum_{j\ge0}N^{\frac{1}{p}}2^{-j\delta}\|\Delta_{\le j}v\|_{L_\mu^2\cap H^{\gamma+\delta+\delta'}}+N^{-\frac{p-1}{p}}2^{-j\sigma}\|\Delta_{\le j}v\|_{L_\mu^2\cap H^{\gamma+\sigma+\sigma'+\frac{d}{2}-\kappa}}\\
&\lesssim\sum_{j\ge0}N^{\frac{1}{p}}2^{-j\delta}\|\Delta_{\le j}v\|_{L_\mu^2\cap H^{\gamma+\delta+\delta'}}+N^{-\frac{p-1}{p}}2^{-j\sigma}2^{\eta j}\|\Delta_{\le j}v\|_{L_\mu^2\cap H^{\gamma+\sigma+\sigma'+\frac{d}{2}-\kappa-\eta}}\\
&\lesssim\sum_{j\ge0}2^{-j\kappa}\|\Delta_{\le j}v\|_{L_\mu^2\cap H^{\gamma+\frac{1}{p}\eta'+\kappa+\delta'}}+2^{-\frac{p-1}{p}\eta'j}2^{-j\sigma}2^{\frac{p-1}{p}\eta'j}\|\Delta_{\le j}v\|_{L_\mu^2\cap H^{\gamma+\sigma+\sigma'+\frac{d}{2}-\kappa-\frac{p-1}{p}\eta'}}\\
% &\lesssim\sum_{j\ge0}2^{-j\kappa}\|\Delta_{\le j}v\|_{H^{\frac{1}{6}\gamma'+\eps+\kappa+\delta'}}+2^{-\frac{5}{6}\gamma'j}2^{-j\sigma}2^{\frac{5}{6}\gamma'j}\|\Delta_{\le j}v\|_{H^{\eps+\sigma+\sigma'+1-\kappa-\frac{5}{6}\gamma'}}\\
&\lesssim\|v\|_{L_\mu^2\cap H^{\gamma+\frac{1}{p}\eta'+\kappa+\delta'}}+\|v\|_{L_\mu^2\cap H^{\gamma+\frac{d}{2}-\kappa-\frac{p-1}{p}\eta'+\sigma+\sigma'}}\\
&\lesssim\|v\|_{L_\mu^2\cap H^{\gamma+\frac{d}{2p}+\eps'}}
\end{align*}
for any $\eps'>0$. For the sum $k>j$, we take
\begin{equation}
\delta>0,\delta'=\frac{1}{p}\gamma'+\kappa,\sigma>0,\sigma'>0,N=2^{\gamma'k},\eta=\frac{p-1}{p}\eta',\eta'=\frac{d}{2}-\kappa
\end{equation}
hence
\begin{align*}
\sum_{k>j}\|\Delta_je^{it\CH^\sharp}\Delta_kv\|_{L^p([0,1],W^{\gamma,q}(\IR^d))}&\lesssim\sum_{k\ge0}\sum_{j<k}N^{\frac{1}{p}}2^{-j\delta}2^{-k\delta'}\|\Delta_kv\|_{L_\mu^2\cap H^{\gamma+\delta+\delta'}}\\
&\hspace{1.5cm}+N^{-\frac{p-1}{p}}2^{-j\sigma}2^{-k\sigma'}\|\Delta_kv\|_{L_\mu^2\cap H^{\gamma+\sigma+\sigma'+\frac{d}{2}-\kappa}}\\
&\lesssim\sum_{k\ge0}2^{\frac{1}{p}\eta'k}2^{-k\delta'}\|\Delta_kv\|_{L_\mu^2\cap H^{\gamma+\delta+\delta'}}\\
&\hspace{1.5cm}+2^{-\frac{p-1}{p}\eta'k}2^{-k\sigma'}2^{k\eta}\|\Delta_kv\|_{L_\mu^2\cap H^{\gamma+\sigma+\sigma'+\frac{d}{2}-\kappa-\eta}}\\
&\lesssim\|v\|_{L_\mu^2\cap H^{\gamma+\frac{1}{p}\eta'+\eps'}}+\|v\|_{L_\mu^2\cap H^{\gamma+\frac{d}{2}-\kappa-\frac{p-1}{p}\eta'+\eps'}}\\
&\lesssim\|v\|_{L_\mu^2\cap H^{\gamma+\frac{d}{2p}+\eps'}}
\end{align*}
for any $\eps'>0$ which completes the proof. The bounds for $S_t$ follows using continuity results on $\Gamma$ and $\Gamma^{-1}$ on $W^{\sigma,p}$ for any $0\le \sigma<3-\frac{d}{2}$.
\end{proof}

\section{Local well-posedness for low regularity initial data}\label{Sec:lowregsolutions}

In this section, the nonlinear equation
\begin{equation}
i\partial_tu=\CH u+f(u)
\end{equation}
is interpreted as a mild formulation associated to the linear propagation $S_t$. For $u=e^Xv$ and initial data $v_0\in L_\mu^2(\IR^d)\cap\Gamma H^2(\IR^d)$, the equation is formaly
\begin{equation}
i\partial_tv=e^{-X}\CH e^Xv+e^{-X}f(e^Xv)
\end{equation}
hence the mild formulation 
\begin{equation}
v_t=S_tv_0-i\int_0^tS_{t-s}e^{-X}f(e^Xv_s)\drm s
\end{equation}
which can be solved using a fixed point formulation. Using that the linear flow satisfies
\begin{equation}
ie^{2X}\partial_tS_t=e^{X}\CH e^XS_t
\end{equation}
in suitable spaces, any such fixed point satisfies
\begin{align}
ie^{2X}\partial_tv_t&=e^X\CH e^XS_tv_0+e^{X}f(e^Xv_t)+\int_0^te^X\CH e^XS_{t-s}e^{-X}f(e^Xv_s)\drm s\\
&=e^X\CH e^Xv_t+e^{X}f(e^Xv_t)
\end{align}
hence it is an energy solutions for $v_0\in L_\mu^2(\IR^d)\cap H^1(\IR^d)$. Since the linear propagator is continuous from $L_\sigma^2(\IR^d)\cap H^\sigma(\IR^d)$ to itself for $\sigma\in(0,1)$, one can search for mild solutions given low regularity initial data. We consider the fixed point map
\begin{equation}
\Phi(v)_t:=S_tv_0-i\int_0^tS_{t-s}e^{-X}f(e^Xv_s)\drm s
\end{equation}
for $v\in\CC([0,T],L_\sigma^2(\IR^d)\cap H^\sigma(\IR^d))$ with $\sigma>0$.

\subsection{Polynomial nonlinearity}

For the polynomial nonlinearity $f(u)=|u|^{m-1}u$ in two dimensions, consider the fixed point map
\begin{equation}
\Phi(v)_t:=S_tv_0-i\int_0^tS_{t-s}e^{(m-1)X}|v_s|^{m-1}v_s\drm s
\end{equation}
for any $v\in L^\infty([0,T],L_\sigma^2(\IR^2)\cap H^\sigma(\IR^2))$ with $\frac{1}{2}<\sigma<1$. We have
\begin{align}
\|\Phi(v)_t\|_{L_\sigma^2\cap H^\sigma}&\lesssim\|v_0\|_{L_\sigma^2\cap H^\sigma}+\int_0^t\|e^{(m-1)X}|v_s|^{m-1}v_s\|_{L_\sigma^2\cap H^\sigma}\drm s\\
&\lesssim\|v_0\|_{L_\sigma^2\cap H^\sigma}+\|e^{(m-1)X}\|_{\CC^\sigma}\int_0^T\|v_s\|_{L^\infty}^{m-1}\|v_s\|_{L_\sigma^2\cap H^\sigma}\drm s\\
&\lesssim\|v_0\|_{L_\sigma^2\cap H^\sigma}+T^{1-\frac{m-1}{p}}\|v\|_{L^p([0,T],L^\infty)}^{m-1}\|v\|_{L^\infty([0,T],L_\sigma^2\cap H^\sigma)}
\end{align}
using Hölder inequality in the last line which also requires $v\in L^p([0,T],L^\infty(\IR^2))$ with $p\ge m-1$. This is where Strichartz inequalities are crucial using the embbeding
\begin{equation}
W^{\gamma,q}(\IR^2)\hookrightarrow L^\infty(\IR^2)
\end{equation}
for $\gamma q>2$ where $W^{s,p}=B_{p,p}^s$. We have
\begin{align}
\|\Phi(v)\|_{L^p([0,T],W^{\gamma,q})}&\lesssim\|v_0\|_{L_\mu^2\cap H^{\gamma+\frac{1}{p}+\kappa}}+\int_0^T\|e^{(m-1)X}|v_t|^{m-1}v_t\|_{L_\mu^2\cap H^{\gamma+\frac{1}{p}+\kappa}}\drm t\\
&\lesssim\|v_0\|_{L_\mu^2\cap H^{\gamma+\frac{1}{p}+\kappa}}+\int_0^T\|v_t\|_{L^\infty}^{m-1}\|v_t\|_{L_\mu^2\cap H^{\gamma+\frac{1}{p}+\kappa}}\drm t\\
&\lesssim\|v_0\|_{L_\mu^2\cap H^{\gamma+\frac{1}{p}+\kappa}}+T^{1-\frac{m-1}{p}}\|v\|_{L^p([0,T],L^\infty)}^{m-1}\|v\|_{L^\infty([0,T],L_\mu^2\cap H^{\gamma+\frac{1}{p}+\kappa})}
\end{align}
for $\mu=\gamma+\frac{1}{p}+\kappa<1$ since $e^{(m-1)X}\in\CC^{1-\kappa}(\IR^2)$ thus we control the norm in the solution space
\begin{equation}
\CS_T^{\sigma,p,q,\gamma}(\IR^2):=L^\infty([0,T],L_\sigma^2\cap H^\sigma)\cap L^p([0,T],W^{\gamma,q})
\end{equation}
for $(p,q)$ a Strichartz pair and $\sigma=\mu=\gamma+\frac{1}{p}+\kappa<1$. We have to optimize the choice of $(p,q)$, that is
\begin{equation}
\gamma<1,\quad \gamma q>2\quad\text{and}\quad\frac{2}{p}+\frac{2}{q}=1
\end{equation}
to have the condition on $\sigma$. We get
\begin{equation}
\sigma>1-\frac{1}{p}
\end{equation}
which gives $\sigma>\frac{1}{2}$ taking $p>2$ as small as possible.

\medskip

\begin{theorem}
Let $\sigma\in(\frac{1}{2},1)$. For any $v_0\in L_\sigma^2(\IR^2)\cap H^\sigma(\IR^2)$ and $(p,q)$ a Strichartz pair such that
\begin{equation}
\sigma>1-\frac{1}{p},
\end{equation} 
there exists a time $T>0$ until which there exists a unique solution
\begin{equation}
v\in\CC([0,T],L_\sigma^2(\IR^2)\cap H^\sigma(\IR^2))\cap L^p([0,T],W^{\gamma,q}(\IR^2))
\end{equation}
to the mild formulation
\begin{equation}
v_t=S_tv_0-i\int_0^tS_{t-s}e^{(m-1)X}|v_s|^{m-1}v_s\drm s
\end{equation}
with $p>m-1$. In particular, there exists a unique mild solution for $\sigma_m<\sigma<1$ with
\begin{equation}
\sigma_m=1-\frac{1}{m-1}
\end{equation}
for any $m\ge2$. Moreover, the solution depends continuously on the initial data $v_0\in L_\sigma^2(\IR^2)\cap H^\sigma(\IR^2)$.
\end{theorem}

\medskip

\begin{proof}
The proof follows from a contraction argument with the bounds explained above. One can choose $R>0$ such that
\begin{equation}
\Phi:B(0,R)_{\CS_T^{\sigma,p,q,\gamma}(\IR^2)}\to B(0,R)_{\CS_T^{\sigma,p,q,\gamma}(\IR^2)}
\end{equation}
is a contraction for $T=T(R)>0$ small enough. Indeed, since
\begin{equation}
\|\Phi(v)\|_{\CS_T^{\sigma,p,q,\gamma}(\IR^2)}\le C\|v_0\|_{L_\sigma^2(\IR^2)\cap H^\sigma(\IR^2)}+CT^{\frac{p-m+1}{p}}\|v\|_{\CS_T^{\sigma,p,q,\gamma}(\IR^2)}^m
\end{equation}
with $C=C(\Xi)>0$, taking $R=2C\|v_0\|_{L_\sigma^2(\IR^2)\cap H^\sigma(\IR^2)}$ and $T\le R^{-m\frac{p}{p-m+1}}$ gives that $\Phi$ sends the ball $B(0,R)$ to itself. Then taking $T$ small enough gives that $\Phi$ is a contraction with similar computations for $p>m-1$.
\end{proof}

\subsection{Hartree nonlinearity}

For the Hartree nonlinearity in three dimensions, consider the fixed point map
\begin{equation}
\Phi(v)_t:=S_tv_0-i\int_0^tS_{t-s}v_s\cdot V_\beta*|e^{X}v_s|^2\drm s
\end{equation}
for any $v\in L^\infty([0,T],L_\sigma^2(\IR^2)\cap H^\sigma(\IR^2))$. Strichartz inequalities from Theorem \ref{thm:Strichartz} gives
\begin{equation}
\|S_tv_0\|_{L^2([0,T],W^{\gamma,6}(\IR^3))}\lesssim\|v_0\|_{L_{\frac{3}{4}+\gamma+\kappa}^2(\IR^3)\cap H^{\frac{3}{4}+\gamma+\kappa}(\IR^3)}
\end{equation}
for $\gamma\in(0,\frac{1}{4})$ with $(2,6)$ a Strichartz pair in three dimensions. We consider the solution space
\begin{equation}
L^\infty([0,T],L_\sigma^2\cap H^\sigma)\cap L^2([0,T],W^{\gamma,6})
\end{equation}
with $\sigma=\frac{3}{4}+\gamma<1$. We have
\begin{equation}
W^{\gamma,6}(\IR^3)\hookrightarrow L^{\frac{6}{1-2\gamma}}(\IR^3)
\end{equation}%$W^{\beta,1}(\IR^3)\hookrightarrow L^{\frac{3}{3-\beta}}(\IR^3)$ for $0\le\beta\le3$
which goes not allow to obtain $L^\infty(\IR^3)$. We have
\begin{align}
\|\Phi(&v)_t\|_{H^\sigma}\lesssim\|v_0\|_{L_\sigma^2\cap H^\sigma}+\int_0^t\|v_s\cdot V_\beta*|e^{X}v_s|^2\|_{L_\sigma^2\cap H^\sigma}\drm s\\
&\lesssim\|v_0\|_{L_\sigma^2\cap H^\sigma}+\int_0^T\big(\|V_\beta*|e^{X}v_s|^2\|_{L^\infty}\|v_s\|_{L_\sigma^2\cap H^\sigma}+\|v_s\|_{L^{\frac{6}{1-2\gamma}}}\|V_\beta*|e^{X}v_s|^2\|_{L_\sigma^2\cap W^{\sigma,\frac{3}{1+\gamma}}}\big)\drm s\\
&\lesssim\|v_0\|_{L_\sigma^2\cap H^\sigma}+\int_0^T\|V_\beta\|_{W^{\beta,1}}\|e^{2X}|v_s|^2\|_{W^{\delta-\beta+\eps,\frac{3}{\delta}}}\|v_s\|_{L_\sigma^2\cap H^\sigma}\drm s\\
&\hspace{2cm}+\int_0^T\|v_s\|_{L^{\frac{6}{1-2\gamma}}}\|V_\beta\|_{W^{\beta,1}}\|e^{2X}|v_s|^2\|_{L_\sigma^2\cap W^{\sigma-\beta,\frac{3}{1+\gamma}}}\drm s\\
&\lesssim\|v_0\|_{L_\sigma^2\cap H^\sigma}+\|v\|_{L^\infty([0,T],L_\sigma^2\cap H^\sigma)}\int_0^T\|v_s^2\|_{W^{\delta-\beta+\eps,\frac{3}{\delta}}}\drm s+\int_0^T\|v_s\|_{L^{\frac{6}{1-2\gamma}}}\|v_s^2\|_{L_\sigma^2\cap W^{\sigma-\beta,\frac{3}{1+\gamma}}}\drm s
\end{align}
for $\eps>0$ small using fractionnal Leibniz rule, Young inequality and assuming that $0\le\delta-\beta<\frac{1}{2}$ as well as $\sigma-\beta<\frac{1}{2}$ since $X\in\CC^{\frac{1}{2}-\kappa}(\IR^3)$. Using the bounds
\begin{equation}
\|v_s^2\|_{L_\mu^2\cap W^{\sigma-\beta,\frac{3}{1+\gamma}}}\lesssim\|v_s\|_{L^{\frac{6}{1-2\gamma}}}\|v_s\|_{L_\mu^2\cap W^{\sigma-\beta,\frac{6}{1+4\gamma}}}
\end{equation}
and
\begin{equation}
\|v_s^2\|_{L_\mu^2\cap W^{\delta-\beta+\eps,\frac{3}{\delta}}}\lesssim\|v_s\|_{L^{\frac{6}{1-2\gamma}}}\|v_s\|_{L_\mu^2\cap W^{\delta-\beta+\eps,\frac{6}{2\delta+2\gamma-1}}},
\end{equation}
we get
\begin{align}
\|\Phi(v)_t\|_{H^\sigma}&\lesssim\|v_0\|_{L_\mu^2\cap H^\sigma}+\|v\|_{L^\infty([0,T],L_\mu^2\cap H^\sigma)}\int_0^T\|v_s\|_{L^{\frac{6}{1-2\gamma}}}\|v_s\|_{L_\mu^2\cap W^{\delta-\beta+\eps,\frac{6}{2\delta+2\gamma-1}}}\drm s\\
&\hspace{2cm}+\int_0^T\|v_s\|_{L^{\frac{6}{1-2\gamma}}}^2\|v_s\|_{L_\mu^2\cap W^{\sigma-\beta,\frac{6}{1+4\gamma}}}\drm s.
\end{align}
At this point, taking $\delta-\beta+\eps=\gamma$ or $\frac{6}{2\delta+2\gamma-1}=6$ gives the condition
\begin{equation}
\beta>1-2\gamma
\end{equation}
which allows to cover all $\beta>\frac{1}{2}$ taking $\gamma>0$ small. For the other norm, Strichartz inequalites from Theorem \ref{thm:Strichartz} gives
\begin{equation}
\|\Phi(v)\|_{L^2([0,T],W^{\gamma,6})}\lesssim\|v_0\|_{L_\mu^2\cap H^{\frac{3}{4}+\gamma+\kappa}}+\int_0^T\|v_t\cdot V_\beta*|e^Xv_s|^2\|_{L_\mu^2\cap H^{\frac{3}{4}+\gamma+\kappa}}\drm t
\end{equation}
for $\mu=\frac{3}{4}+\gamma+\kappa<1$ which is bounded as before for $\sigma=\frac{3}{4}+\gamma+\kappa$. Following the previous proof, one can find a contraction using these bounds to obtain the following Theorem.

\medskip

\begin{theorem}
Let $\sigma=\frac{3}{4}+\gamma$ with $\gamma\in(0,\frac{1}{4})$. For any $v_0\in L_\sigma^2(\IR^2)\cap H^\sigma(\IR^2)$ and
\begin{equation}
\beta>1-2\gamma,
\end{equation} 
there exists a time $T>0$ until which there exists a unique solution
\begin{equation}
v\in\CC([0,T],L_\sigma^2(\IR^2)\cap H^\sigma(\IR^2))\cap L^2([0,T],W^{\gamma,6}(\IR^2))
\end{equation}
to the mild formulation
\begin{equation}
v_t=S_tv_0-i\int_0^tS_{t-s}v_s\cdot V_\beta*|e^Xv_s|^2\drm s
\end{equation}
where $V_\beta\in W^{\beta,1}$. Moreover, the solution depends continuously on the initial data $v_0\in L_\sigma^2(\IR^2)\cap H^\sigma(\IR^2)$. In particular, there exists a unique mild solution for $\sigma_\beta<\sigma<1$ with
\begin{equation}
\sigma_\beta=\frac{5}{4}-\frac{\beta}{2}
\end{equation}
for any $\beta>\frac{1}{2}$.
\end{theorem}

\subsection{Uniqueness of energy solutions}

Our definition of energy solutions is a priori different than mild solutions, one has that a mild solution in the energy space is indeed an energy solution however the opposite is not immediate. Uniqueness of energy solutions direclty follows from Strichartz inequalities with Gronwall's lemma. In two dimensions, we consider any polynomial nonlinearity $f(u)=|u|^{m-1}u$ with $m\ge1$.

\medskip

\begin{corollary}
Let $0<\mu\le 1$ and $m\ge1$. For $u_0\in e^X(H^1(\IR^2)\cap L_\mu^2(\IR^2))$, consider the unique solution $u_\eps$ to the regularized equation associated to \eqref{NLS}. Then there exists a unique function $u\in\CC(\IR,e^X(H^1(\IR^2)\cap L_\mu^2(\IR^2)))$ such that
\begin{equation}
\lim_{\eps\to0}\ \sup_{t\in[-T,T]}\|e^{-X_\eps}u_\eps(t)-e^{-X}u(t)\|_{H^1(\IR^2)\cap L_\mu^2(\IR^2)}=0
\end{equation}
for any $T>0$.
\end{corollary}

\medskip

\begin{proof}
Consider the unique solution to the regularized equation
\begin{equation}
ie^{2X_\eps}\partial_tv_\eps=e^{X_\eps}\CH_\eps e^{X_\eps}v_\eps+e^{(m+1)X_\eps}|v_\eps|^{m-1}v_\eps
\end{equation}
with $v_\eps(0)=v_0\in H^1(\IR^2)\cap L_\mu^2(\IR^2)$. To complete the proof, we need to show that $v_\eps$ has a unique limit as $\eps$ goes to $0$. Let $\eps>\eps'>0$ and consider $w=u_{\eps}-u_{\eps'}$. We have
\begin{equation}
ie^{2X_\eps}\partial_tw=e^{X_\eps}\CH_\eps e^{X_\eps}w+(e^{X_\eps'}\CH_{{\eps'}}e^{X_{\eps'}}-e^{X_\eps}\CH_\eps e^{X_\eps})v_{\eps'}+e^{(m+1)X_\eps}|v_\eps|^{m-1}v_\eps-e^{(m+1)X_{\eps'}}|v_{\eps'}|^{m-1}v_{\eps'}
\end{equation}
which we multply by $\overline{w}(t,x)$ and integrate in space. We get
\begin{align}
\frac{1}{2}\frac{\drm}{\drm t}\int_{\IR^2}|w(t,x)|^2\drm x&=\text{Im}\Big(\int_{\IR^2}(e^{X_{\eps'}}\CH_{\eps'}e^{X_{\eps'}}v_{\eps'}-e^{X_\eps}\CH_\eps e^{X_\eps}v_{\eps'})(t,x)\overline{w}(t,x)\drm x\\
&\quad+\Big(\int_{\IR^2}(e^{(m+1)X_\eps}|v_\eps|^{m-1}v_\eps-e^{(m+1)X_{\eps'}}|v_{\eps'}|^{m-1}v_{\eps'})(t,x)\overline{w}(t,x)\drm x\Big)\\
&\lesssim\|(e^{X_{\eps'}}\CH_{\eps'}e^{X_{\eps'}}-e^{X_\eps}\CH_\eps e^{X_\eps})v_{\eps'}\|_{H^{-1}(\IR^2)}\|w\|_{H^1(\IR^2)}\\
&\quad+\Big(1+\|v_\eps\|_{L^\infty(\IR^2)}^{m-1}+\|v_{\eps'}(t)\|_{L^\infty(\IR^2)}^{m-1}\Big)\int_{\IR^2}|w(t,x)|^2\drm x
\end{align}
hence Gronwall's lemma
\begin{equation}
\int_{\IR^2}|w(T,x)|^2\drm x\lesssim \big(M_{\eps,\eps'}+\int_{\IR^2}|w(0,x)|^2\drm x\big)e^{c\int_0^T(\|v_\eps\|_{L^\infty(\IR^2)}^{m-1}+\|v_{\eps'}(t)\|_{L^\infty(\IR^2)}^{m-1})\drm t}
\end{equation}
for any $T>0$ and
\begin{equation}
M_{\eps,\eps'}=\|(e^{X_{\eps'}}\CH_{\eps'}e^{X_{\eps'}}-e^{X_\eps}\CH_\eps e^{X_\eps})v_{\eps'}\|_{H^{-1}(\IR^2)}(\|v_\eps\|_{H^1(\IR^2)}+\|v_{\eps'}\|_{H^1(\IR^2)})
\end{equation}
since $\|w\|_{H^1(\IR^2)}\le\|v_\eps\|_{H^1(\IR^2)}+\|v_{\eps'}\|_{H^1(\IR^2)}$. Since
\begin{equation}
e^{X_{\eps'}}\CH_{\eps'}e^{X_{\eps'}}-e^{X_\eps}\CH_\eps e^{X_\eps}=-\nabla(e^{2X_{\eps'}}-e^{2X_\eps})\nabla+e^{2X_{\eps'}}Y_{\eps'}-e^{2X_\eps}Y_\eps
\end{equation}
and $v_{\eps'}$ being uniformly bounded in $H^1(\IR^2)\cap L_\mu^2(\IR^2)$, we get that $(e^{X_\eps'}\CH_{\eps'}e^{X_\eps'}-e^{X_\eps}\CH_\eps e^{X_\eps})v_{\eps'}$ converges to $0$ in $H^{-1}(\IR^2)$ as $\eps>\eps'>0$ goes to $0$ thus $M_{\eps,\eps'}$ goes to $0$. With our Strichartz inequalities, $v_\eps$ is uniformly bounded in $L^p([0,T],L^\infty(\IR^2))$ thus and we have
\begin{equation}
w(0)=(e^{X_\eps}-e^{X_{\eps'}})v_0
\end{equation}
which converges to $0$ in $L^2(\IR^2)$ as $\eps>\eps'>0$ goes to $0$. This proves that $(v_\eps)_{\eps>0}$ is a Cauchy sequence in $L^2(\IR^2)$ and complete the proof.
\end{proof}

In three dimensions, we consider \ref{Hartree} equation where $f(u)=u\cdot V_\beta*|u|^2$ with $V_\beta\in W^{\beta,1}$ for $\beta>\frac{1}{2}$. Following the previous proof, we also get uniqueness of energy solutions for \eqref{Hartree} equation.

\medskip

\begin{corollary}
Let $0<\mu\le 1$ and $\beta>\frac{1}{2}$. For $u_0\in e^X(H^1(\IR^2)\cap L_\mu^2(\IR^2))$, consider the unique solution $u_\eps$ to the regularized equation associated to \eqref{Hartree}. Then there exists a unique function $u\in\CC(\IR,e^X(H^1(\IR^2)\cap L_\mu^2(\IR^2)))$ such that
\begin{equation}
\lim_{\eps\to0}\ \sup_{t\in[-T,T]}\|e^{-X_\eps}u_\eps(t)-e^{-X}u(t)\|_{H^1(\IR^2)\cap L_\mu^2(\IR^2)}=0
\end{equation}
for any $T>0$.
\end{corollary}

\appendix

\section{Function spaces} \label{Sec:FunctionalSpaces}

Since the law of the white noise is invariant by translation, it does not decay at infinity. To deal with this, we work in weighted Lebesgue and Besov spaces. For any $\mu\in\R$ and $p\in[1,\infty]$, we consider
\begin{equation*}
\|u\|_{L_\mu^p(\R^d)}=\Big(\int_{\R^d}\langle x\rangle^\mu|f(x)|^p\drm x\Big)^{\frac{1}{p}}
\end{equation*}
with $\langle x\rangle=\sqrt{1+|x|^2}$. An important tool is the Paley-Littlewood decomposition
\begin{equation*}
u=\sum_{n\ge0}\Delta_nu
\end{equation*}
where 
\begin{equation*}
\big(\Delta_nu\big)(x):=2^{d(n-1)}\int_{\R^d}\chi\big(2^{n-1}(x-y)\big)u(y)\drm y
\end{equation*} 
with $\chi\in\CS(\R^d)$ and $\supp\ \widehat\chi\subset\{\frac{1}{2}\le|z|\le 2\}$ for $n\ge1$ and
\begin{equation*}
\big(\Delta_0u\big)(x):=\int_{\R^d}\chi_0(x-y)u(y)\drm y
\end{equation*}
with $\chi_0\in\CS(\R^d)$ and $\supp\ \widehat\chi_0\subset\{|z|\le 1\}$. Most of the following definitions and properties can be found in \cite[Section 4]{edmunds1996}. See also \cite{DW} and references therein, in particular the book \cite{BCD}.

\medskip

\begin{definition}
Let $\mu,\alpha\in\R$ and $p,q\in[1,\infty)$. The weighted Besov space $\CB_{p,q,\mu}^\alpha$ is the set of distribution $u\in\CS'(\R^d)$ such that
\begin{equation*}
\|u\|_{\CB_{p,q,\mu}^\alpha}:=\Big(\sum_{n\ge0}2^{\alpha n q}\|\Delta_nu\|_{L_\mu^p(\R^d)}^q\Big)^{\frac{1}{q}}<\infty.
\end{equation*}
\end{definition}

\medskip

For $p=q=2$, one recovers the usual weighted Sobolev spaces $\CH_\mu^\alpha=\CB_{2,2,\mu}^\alpha$ with
\begin{equation*}
\|u\|_{\CH_\mu^\alpha(\R^d)}=\|(\SF^{-1}\langle\cdot\rangle^\alpha\SF)u\|_{L_\mu^2(\R^d)}.
\end{equation*}
We also denote the case $p=q=\infty$ as $\CC_\mu^\alpha=\CB_{\infty,\infty,\mu}^\alpha$ which corresponds to the usual weighted Hölder spaces for $\alpha\in\R_+\backslash\N$. Moreover, there exist constants $C_1,C_2>0$ depending on the spaces parameters such that
\begin{equation*}
C_1\|\langle\cdot\rangle^\mu u\|_{\CB_{p,q,0}^\alpha}\le\|u\|_{\CB_{p,q,\mu}^\alpha}\le C_2\|\langle\cdot\rangle^\mu u\|_{\CB_{p,q,0}^\alpha},
\end{equation*}
and $\CB_{p,q,0}^\alpha$ corresponds to the usual Besov spaces hence weighted Besov spaces satisfy the following embeddings.

\medskip

\begin{lemma} \label{besov_embeddings}% \textbf{(Besov embeddings).} \label{besov_embeddings}\\
Let $p_1,p_2,q_1,q_2\in[1,\infty]$ and $\mu,\mu'\in\R$ such that $p_1\le p_2,q_1\le q_2$ and $\mu_1\ge\mu_2$. For all $\alpha\in \R$, one has the Besov embeddings
\begin{equation*}
\mathcal{B}^\alpha_{p_1,q_1,\mu_1}(\R^d) \subset  \mathcal{B}^{\alpha-d\left(\frac{1}{p_1}-\frac{1}{p_2} \right)}_{p_2,q_2,\mu_2}(\R^d).
\end{equation*}
as well as the Sobolev embeddings
\begin{equation*}
\forall\alpha\ge \frac{d}{2}-\frac{d}{p},\quad \CH^\alpha_{\mu}(\R^d) \subset  L^p_{\mu}(\R^d)
\end{equation*}
for $p\in[2,\infty]$. Finally, the embedding
\begin{equation*}
\CH_{\mu_1}^{\alpha_1}(\R^d)\hookrightarrow\CH_{\mu_2}^{\alpha_2}(\R^d)
\end{equation*}
is compact for $\alpha_1>\alpha_2$ and $\mu_1>\mu_2$.
\end{lemma}

\medskip

The following lemma is a useful interpolation estimate that we shall use in the following, and which can be found in \cite[Theorem 3.8]{sickel2014}.

\medskip

\begin{lemma} \label{interpolation_estimate_Besov}
Let $p_0,p_1,q_0,q_1 \in[1,\infty]$ and $\alpha_0,\alpha_1,\mu_0,\mu_1 \in \R$ and $p,q,\alpha,\mu$ such that
\begin{equation*} 
\frac{1}{p}=\frac{1-\theta}{p_0}+\frac{\theta}{p_1},\quad \frac{1}{q}=\frac{1-\theta}{q_0}+\frac{\theta}{q_1},
\end{equation*}
and
\begin{equation*}
\alpha=(1-\theta)\alpha_0+\theta\alpha_1, \quad  \mu=(1-\theta)\mu_0 + \theta \mu_1
\end{equation*}
for $\theta \in[0,1]$. Then there exists $C>0$ such that
\begin{equation*} 
\|u\|_{\mathcal{B}^\alpha_{p,q,\mu}(\R^d)} \leq C \|u\|_{\mathcal{B}^{\alpha_0}_{p_0,q_0,\mu_0}(\R^d)}^{1-\theta} \|u\|_{\mathcal{B}^{\alpha_1}_{p_1,q_1,\mu_1}(\R^d)}^{\theta}.
\end{equation*}
In particular, we have
\begin{equation*}
\|u\|_{\CH_\mu^\alpha}\le C\|u\|_{\CH_{\mu_0}^{\alpha_0}}^{1-\theta}\|u\|_{\CH_{\mu_1}^{\alpha_1}}^\theta.
\end{equation*}
\end{lemma}

\medskip

In general, one can only multiply a distribution by a test function. The following lemma gives a product rule in weighted Besov space, as a generalisation of Young condition, see~\cite[Section 2.8.2]{triebel1983}.

\medskip

\begin{lemma}\label{product_estimate_Besov}
Let $\alpha_1,\alpha_2 \in \R$ such that $\alpha_1+\alpha_2>0$. Let $\mu_1,\mu_2 \in \R$ with $\mu=\mu_1+\mu_2$ and $p_1,p_2\in[1,\infty]$ with $\frac{1}{p}=\frac{1}{p_1}+\frac{1}{p_2}$. Then for any $\kappa >0$, there exists a constant $C>0$ such that
\begin{equation*} 
\|uv\|_{\mathcal{B}_{p,p,\mu}^{\alpha-\kappa}(\R^d)} \leq C \| u \|_{\mathcal{B}_{p_1,p_1,\mu_1}^{\alpha_1}(\R^d)} \| v \|_{\mathcal{B}_{p_2,p_2,\mu_2}^{\alpha_2}(\R^d)}
\end{equation*}
where $\alpha=\min(\alpha_1,\alpha_2)$. In the case where $p=p_1$, one can take $\kappa=0$. We also have
\begin{equation*} 
\|uv\|_{\CW^{\alpha,p}(\R^d)} \leq C \| u \|_{\CW^{\alpha,p_1}(\R^d)} \| v \|_{L^{p_2}(\R^d)}+ C \| u \|_{L^{p_1'}(\R^d)} \| v \|_{\CW^{\alpha,p_2'}(\R^d)}
\end{equation*}
for $\alpha\in[0,1]$ and $\frac{1}{p}=\frac{1}{p_1}+\frac{1}{p_2}=\frac{1}{p_1'}+\frac{1}{p_2'}$.
\end{lemma}

\medskip

As we will control energy associated to dispersive PDEs, the following duality result from in weighted Besov spaces from \cite[Theorem 2.11.2]{triebel1983} will be useful.

\medskip

\begin{lemma}\label{duality_estimate_Besov}
Let $\alpha,\mu \in \R$ and $p,q \in [1,\infty ]$. Then there exists a constant $C>0$ such that
\begin{equation*} 
\Big| \int_{\R^d} u(x)v(x)\dd x \Big| \leq C \|u\|_{\mathcal{B}_{p,q,\mu}^\alpha(\R^d)} \|v\|_{\mathcal{B}_{p',q',-\mu}^{-\alpha}(\R^d)},
\end{equation*}
where $\frac{1}{p}+\frac{1}{p'}=1$ and  $\frac{1}{q}+\frac{1}{q'}=1$ for any $u,v\in\CS(\IR^d)$.
\end{lemma}

\section{Stochastic bounds and renormalization}\label{Sec:stochastic_bounds}

The Gaussian spatial white noise $\xi$ belongs to $\CC_{-\delta}^{-\frac{d}{2}-\kappa}$ for any $\kappa,\delta>0$, that is spatial growth at infinity of order $\langle x\rangle^\delta$. We consider
\begin{equation}
X_1:=(1-\Delta)^{-1}\xi\in\CC_{-\delta}^{2-\frac{d}{2}-\kappa}
\end{equation}
such that $\nabla X_1$ is a distribution for $d\in\{2,3\}$. Following usual renormalization, its Wick square can be defined as follows given $X_{1,\eps}=(1-\Delta)^{-1}\xi_\eps$ where $(\xi_\eps)_{\eps>0}$ is a regularization of the noise, see for example \cite{HairerLabbe15}.

\medskip

\begin{proposition}
There exists a distribution $\Wick{|\nabla X_1|^2}\in\CC_{-\delta}^{2-d-\kappa}(\IR^d)$ such that
\begin{equation}
\Wick{|\nabla X_1|^2}=\lim_{\eps\to0}\big(|\nabla X_{1,\eps}|^2-c_\eps^1\big)
\end{equation}
in probability in the space $\CC_{-\delta}^{2-d-\kappa}(\IR^d)$ for any $\kappa,\delta>0$ with $c_\eps^1=\IE\big[|\nabla X_{1,\eps}|^2\big]$ a diverging constant.
\end{proposition}

\medskip

As introduced by Gubinelli and Hofmanovà \cite{GubinelliHofmanova19}, one can consider the decomposition
\begin{equation}
f=\SU_{\le}f+\SU_>f
\end{equation}
for any distribution $f$ where $\SU_{\le}f$ and $\SU_>f$ are projected respectively on low or high frequency according to an inhomogeneous spatial decomposition. Precisely, we have
\begin{equation}
\SU_{\le}f=\sum_{k\ge0}w_k\Delta_{\le L_k}f\quad\text{and}\quad\SU_>f=\sum_{k\ge0}w_k\Delta_{>L_k}f
\end{equation}
given $(w_k)_{k\ge0}$ a smooth partition of the unity, $(L_k)_{k\ge0}\subset[-1,+\infty)$ a suitable sequence with $\Delta_{\le L}$ and $\Delta_{>L}$ the sum of the Paley-Littlewood projectors $\Delta_j$ respectively for $j\le L$ and $j>L$. The main properties for the projection used here are recalled in the following result, that is Lemma $2.4$ from \cite{GubinelliHofmanova19}. To obtain the following lemma, consider with their notation
\begin{equation}
\alpha=-\alpha,\quad\gamma=\delta_1,\quad\delta=\delta_2,\quad a=\frac{\mu}{m}-\delta,\quad b=-\frac{\mu}{m}-\gamma
\end{equation}
with $\rho(x)=\langle x\rangle^{\mu}$.

\medskip

\begin{lemma}\label{Lem:LocalizationProjectors}
For any $\alpha,\mu\in\IR$ and $m,\delta_1,\delta_2>0$, we have
\begin{equation}
\|\SU_{\le} f\|_{\CC_{-\mu-m\delta_1}^{\alpha+\delta_1}}\lesssim\|f\|_{\CC_{-\mu}^\alpha}\quad\text{and}\quad\|\SU_>f\|_{\CC_{-\mu+m\delta_2}^{\alpha-\delta_2}}\lesssim\|f\|_{\CC_{-\mu}^\alpha}
\end{equation}
where the constant depends on the parameters $\alpha,\mu,\delta_1,\delta_2$.
\end{lemma}

\medskip

This lemma will be used with $m=\frac{2\mu}{\delta_2}$ for $\mu,\delta_1,\delta_2>0$, that is
\begin{equation}
\|\SU_{\le} f\|_{\CC_{-\mu(1+2\frac{\delta_1}{\delta_2})}^{\alpha+\delta_1}}\lesssim\|f\|_{\CC_{-\mu}^\alpha}\quad\text{and}\quad\|\SU_>f\|_{\CC_{\mu}^{\alpha-\delta_2}}\lesssim\|f\|_{\CC_{-\mu}^\alpha}
\end{equation}
where all the spatial weight are abitrary small for $\mu>0$ small and fixed $\delta_1,\delta_2>0$. In this work, we use the localization operators $\SU_{\le}$ and $\SU_{>}$ to deal with the singularity of the Wick product and the growth at infinity in two different terms. Taking $\delta_1=\eta>0$ and $\delta_2=\kappa$, we have
\begin{equation}
X_1=X_1^\le+X_1^>
\end{equation}
with $X_1^\le=\SU_{\le}X_1\in\CC_{-\mu(1+\frac{\eta}{\kappa})}^{2-\frac{d}{2}-\kappa+\eta}(\IR^d)$ and $X_1^>=\SU_> X_1\in\CC_{\mu}^{2-\frac{d}{2}-2\kappa}(\IR^d)$. Since $\SU_>X_1$ has the same regularity as $X_1$, the square of its gradient has to be renormalizard. We have
\begin{equation}
\Wick{|\nabla X_1|^2}=|\nabla X_1^\le|^2+2\nabla X_1^\le\cdot\nabla X_1^>+\Wick{|\nabla X_1^>|^2}
\end{equation}
where $|\nabla X_1^\le|^2$ is well-defined for $1-\frac{d}{2}-\kappa+\eta>0$ and the cross term is well-defined if
\begin{equation}
2-d-3\kappa+\eta>0
\end{equation}
which imposes a condition on $\eta$. Taking $\eta=d-2+4\kappa$ gives 
\begin{equation}
X_1^\le\in\CC_{-\mu(1+\frac{d-2+4\kappa}{\kappa})}^{\frac{d}{2}+3\kappa}(\IR^d)\quad\text{and}\quad X_1^>\in\CC_{\mu}^{2-\frac{d}{2}-2\kappa}(\IR^d)
\end{equation}
hence all the terms in the previous expression are well-defined while the negative weight can be arbitrary small taking $\mu>0$ small such as $\mu=\kappa^2$. While $X_1^>$ encodes the same singularity as $X_1$, it has decay at infinity hence
\begin{equation}
e^{cX_1^>}\in L^\infty(\IR^d)
\end{equation}
for any $c\in\IR$ and the norm
\begin{equation}
\|ue^{X_1^>}\|_{L^2(\IR^d)}^2=\int_{\IR^d}|u(x)|^2e^{2X_1^>(x)}\drm x
\end{equation}
is equivalent to the unweighted $L^2(\IR^d)$ norm. In three dimensions, one needs to renormalize more terms for $\CH$ to make sense. We consider
\begin{align}
X_1&:=(1-\Delta)^{-1}\xi\in\CC_{-\delta}^{\frac{1}{2}-\kappa}(\IR^3),\\
X_2&:=(1-\Delta)^{-1}(\Wick{|\nabla X_1|^2})\in\CC_{-\delta}^{1-\kappa}(\IR^3),\\
X_3&:=2(1-\Delta)^{-1}(\Wick{\nabla X_1\cdot\nabla X_2})\in\CC_{-\delta}^{\frac{3}{2}-\kappa}(\IR^3)
\end{align}
for $\kappa,\delta>0$. With the previous renormalization result, the field $X_2$ is defined as the limit of
\begin{equation}
X_{2,\eps}:=(1-\Delta)^{-1}\big(|\nabla X_{1,\eps}|^2-c_{1,\eps}\big)
\end{equation}
in $\CC_{-\delta}^{1-\kappa}(\IR^3)$ as $\eps$ goes to $0$. We have the following result that assures that $X_3$ also exists as a limit of the regularized product, it was proved on the torus $\IT^3$ in Section $9$ from \cite{GubinelliPerkowski2017} and can be adapted to weighted full space as for the previous result as for the other terms. One also has to renormalized $|\nabla X_2|^2$ for our purpose.

\medskip

\begin{proposition}
There exists distributions $\Wick{\nabla X_1\cdot\nabla X_2}\in\CC_{-\delta}^{-\frac{1}{2}-\kappa}(\IR^3)$ and $\Wick{|\nabla X_2|^2}\in\CC_{-\delta}^{-\kappa}(\IR^3)$ such that
\begin{equation}
\Wick{\nabla X_1\cdot\nabla X_2}=\lim_{\eps\to0}(\nabla X_{1,\eps}\cdot\nabla X_{2,\eps})
\end{equation}
in probability in the space $\CC_{-\delta}^{-\frac{1}{2}-\kappa}(\IR^3)$ and
\begin{equation}
\Wick{|\nabla X_2|^2}=\lim_{\eps\to0}(|\nabla X_{2,\eps}|^2-c_\eps^2)
\end{equation}
in probability in the space $\CC_{-\delta}^{-\kappa}(\IR^3)$ for any $\kappa,\delta>0$ with $c_\eps^2=\IE\big[|\nabla X_{2,\eps}|^2\big]$ a diverging quantity.
\end{proposition}

\medskip

Formally, we have
\begin{equation}
|\nabla(X_1+X_2+X_3)|^2=|\nabla X_1|^2+2\nabla X_1\cdot\nabla(X_2+X_3)+|\nabla X_2|^2+2\nabla X_2\cdot\nabla X_3+|\nabla X_3|^2
\end{equation}
however this involves singular products in three dimensions, namely $|\nabla X_1|^2,|\nabla X_2|^2$ and $2\nabla X_1\cdot\nabla X_2$ since $X_1\in\CC_{-\delta}^{\frac{1}{2}-\kappa}(\IR^3)$ and $X_2\in\CC_{-\delta}^{1-\kappa}(\IR^3)$. With the renormalization procedure, we have the Wick product
\begin{equation}
\Wick{|\nabla(X_1+X_2+X_3)|^2}=\Wick{|\nabla X_1|^2}+2\Wick{\nabla X_1\cdot\nabla X_2}+2\nabla X_1\cdot\nabla X_3+\Wick{|\nabla X_2|^2}+2\nabla X_2\cdot\nabla X_3+|\nabla X_3|^2
\end{equation}
where we also need to add a localization operator. For the first term, we already have
\begin{equation}
\Wick{|\nabla X_1|^2}=|\nabla X_1^\le|^2+2\nabla X_1^\le\cdot\nabla X_1^>+\Wick{|\nabla X_1^>|^2}
\end{equation}
where the first two terms are well-defined since
\begin{equation}
X_1^\le\in\CC_{-\mu(1+\frac{1+4\kappa}{\kappa})}^{\frac{3}{2}+3\kappa}(\IR^3)\quad\text{and}\quad X_1^>\in\CC_{\mu}^{\frac{1}{2}-2\kappa}(\IR^3)
\end{equation}
hence this gives a definition of $\Wick{|\nabla X_1^>|^2}$ without additional renormalization procedure. For the second term we have
\begin{equation}
\Wick{\nabla X_1\cdot\nabla X_2}=\nabla X_1^\le\cdot\nabla X_2^\le+\nabla X_1^\le\cdot\nabla X_2^>+\nabla X_1^>\cdot\nabla X_2^\le+\Wick{\nabla X_1^>\cdot\nabla X_2^>}
\end{equation}
where again only the last term is singular. Indeed, we have
\begin{align}
\nabla X_1^\le&\in\CC_{\mu}^{-\frac{1}{2}-2\kappa}(\IR^3),\\
\nabla X_1^>&\in\CC_{-\mu(1+2\frac{1+4\kappa}{\kappa})}^{\frac{1}{2}+3\kappa}(\IR^3),\\
\nabla X_2^\le&\in\CC_{\mu}^{-2\kappa}(\IR^3),\\
\nabla X_2^>&\in\CC_{-\mu(1+2\frac{\frac{1}{2}+4\kappa}{\kappa})}^{\frac{1}{2}+3\kappa}(\IR^3)
\end{align}
for a suitable choice of parameters with Lemma \ref{Lem:LocalizationProjectors}. Finally, we have
\begin{equation}
\Wick{|\nabla X_2|^2}=|\nabla X_2^\le|^2+2\nabla X_2^\le\cdot\nabla X_2^>+\Wick{|\nabla X_2^>|^2}
\end{equation}
which is also well-defined.

\section{Construction and transformation of $\CH$}\label{sec:construction-AH} 

The Anderson Hamiltonian is the operator
\begin{equation}
\CH=-\Delta+\xi
\end{equation}
where the white noise $\xi$ belongs $\CC_{-\delta}^{-\frac{d}{2}-\kappa}(\IR^d)$ for any $\kappa,\delta>0$. We do not describe the spectral properties of this random operator, see \cite{HsuLabbe25,Ueki25} and references therein as well as Remark \ref{Remark:linearFlow}. This operator is a singular stochastic operator for $d\in\{2,3\}$, we use an exponential transform to study the associated equation. Given $u=e^Xv$, we formally have
\begin{equation}
e^X\CH e^Xv=-(\nabla e^{2X}\nabla)v+e^{2X}(\xi-\Delta X-|\nabla X|^2)v
\end{equation}
which should be better behaved than $\CH$ given a random field $X$ such that $\xi-\Delta X$ is smoother than $\xi$. This requires $X$ to be locally of Hölder regularity $2-\frac{d}{2}-\kappa$ hence the square $|\nabla X|^2$ is ill-defined for $d\ge 2$, this falls in the range of singular dispersive SPDEs and requires a probabilistic renormalization procedure. The operator makes sense for $\xi_\eps$ a smooth regularization of the noise and we want to construct $X_\eps$ such that
\begin{equation}
Y_\eps:=\xi_\eps-\Delta X_\eps-|\nabla X_\eps|^2-c_\eps
\end{equation}
converges to a limit $Y\in\CC_{-\delta}^{1-\frac{d}{2}-\kappa}$ with $c_\eps$ a diverging quantity, this is the renormalization procedure explained above. We first describe the exponential transform in two and three dimensions and then introduce a paracontrolled expansion to obtain a finer description. In both cases, we have
\begin{equation}\label{EqExpressionH}
e^X\CH e^Xv=-\nabla(e^{2X}\nabla v)+e^{2X}Yv
\end{equation}
for any $v\in H_\delta^{\frac{d}{2}-1+2\kappa}(\IR^d)$ where the definition of $X$ and $Y$ is more involved in three dimensions.

\medskip

{\bf Construction in 2D.} In two dimensions, we consider 
\begin{equation}
X_1=-(1-\Delta)^{-1}\xi\in\CC_{-\delta}^{1-\kappa}(\IR^2)
\end{equation}
and
\begin{equation}
X:=\SU_{>}X_1\in \CC_{\delta}^{1-2\kappa}(\IR^2)
\end{equation}
for any $\delta,\kappa>0$. The choice $X_1$ is natural to ensure that $\Delta X_1-\xi$ is smooth while the projection $\SU_>$ is used to localize in space the transformation. The term $|\nabla X|^2$ is defined as a Wick product and we have
\begin{equation}
\Wick{|\nabla X|^2}=\Wick{|\nabla X_1|^2}-2\nabla \SU_{\le}X_1\cdot\nabla X-|\SU_{\le}X_1|^2
\end{equation}
as explained above. We also have
\begin{align}
\xi-\Delta X&=\xi+\Delta\SU_{>}(1-\Delta)^{-1}\xi\\ 
&=\xi+\Delta(1-\Delta)^{-1}\xi-\Delta\SU_{\le}(1-\Delta)^{-1}\xi\\
&=(1-\Delta)^{-1}\xi-\Delta\SU_{\le}(1-\Delta)^{-1}\xi\\
&=-X_1-\Delta\SU_{\le}X_1
\end{align}
using $\SU_{>}=\text{Id}-\SU_{\le}$ thus we get
\begin{equation}
Y=-X_1-\Delta\SU_{\le}X_1-\Wick{|\nabla X_1|^2}+2\nabla \SU_{\le}X_1\cdot\nabla X+|\SU_{\le}X_1|^2\in\CC_{-\delta}^{-\kappa}(\IR^d).
\end{equation} 
In the end, we have
\begin{equation}	
e^X\CH e^Xv=-\nabla(e^{2X}\nabla v)+e^{2X}Yv
\end{equation}
thus the associated quadratic form is
\begin{equation}
\langle\CH u,u\rangle=\int_{\IR^2}|\nabla v(x)|^2e^{2X(x)}\drm x+\int_{\IR^2}|v(x)|^2Y(x)e^{2X(x)}\drm x
\end{equation}
for any $u=e^Xv$ with $v\in H_{\delta}^\kappa(\IR^2)$, the second term being interpreted as a duality bracket as $Y$ can not be evaluated pointwise. Since $X\in\CC_{\delta}^{1-2\kappa}(\IR^2)\subset L^\infty(\IR^2)$, we have
\begin{equation}
c_1\|w\|_{L^2(\IR^2)}^2\le\int_{\IR^2}|w(x)|^2e^{2X(x)}\drm x\le c_2\|w\|_{L^2(\IR^2)}^2
\end{equation}
for $c_1,c_2>0$ random, this is the main motivation for the localization operator $\SU_{>}$.

\medskip

{\bf Construction in 3D.} In three dimensions, we have $X_1\in\CC_{-\delta}^{\frac{1}{2}-\kappa}(\IR^3)$ and $|\nabla X_1|^2$ is still defined as a Wick product however additionnals terms are required for the renormalization, see for example \cite{GUZ,DeVecchiJiZachhuber25} and references therein. We defined above
\begin{align}
X_2&=(1-\Delta)^{-1}(\Wick{|\nabla X_1|^2})\in\CC_{-\delta}^{1-\kappa}(\IR^3),\\
X_3&=2(1-\Delta)^{-1}(\Wick{\nabla X_1\cdot\nabla X_2})\in\CC_{-\delta}^{\frac{3}{2}-\kappa}(\IR^3)
\end{align}
and we consider
\begin{equation}
X=\SU_{>}(X_1+X_2+X_3)\in\CC_{\delta}^{\frac{1}{2}-\kappa}(\IR^3)
\end{equation}
for any $\delta,\kappa>0$. Using similar computations as for $d=2$ for localization operators, we have
\begin{align}
% \Delta\SU_>(1-\Delta)^{-1}Z=-Z+(1-\Delta)^{-1}Z+\Delta\SU_\le(1-\Delta)^{-1}Z
-\Delta X&=-(1-\Delta)X+X+\Delta\SU_\le(X_1+X_2+X_3)\\
&=-\xi-\Wick{|\nabla X_1|^2}-2\Wick{\nabla X_1\cdot\nabla X_2}+X+\Delta\SU_\le(X_1+X_2+X_3)
\end{align}
where $X+\Delta\SU_\le(X_1+X_2+X_3)$ is a nicely behaved remainder. We have the Wick product
\begin{align}
\Wick{|\nabla X|^2}=\Wick{|\nabla X_1^>|^2}+2\Wick{\nabla X_1^>\cdot\nabla X_2^>}+2\nabla X_1^>\cdot\nabla X_3^>+\Wick{|\nabla X_2^>|^2}+2\nabla X_2^>\cdot\nabla X_3^>+|\nabla X_3^>|^2
\end{align}
where 
\begin{equation}
\Wick{|\nabla X_i^>|^2}=\Wick{|\nabla X_i|^2}-|\nabla X_i^\le|^2-2\nabla X_i^\le\cdot\nabla X_i^>
\end{equation}
for $i\in\{1,2\}$ and
\begin{equation}
\Wick{\nabla X_1^>\cdot\nabla X_2^>}=\Wick{\nabla X_1\cdot\nabla X_2}-\nabla X_1^\le\cdot\nabla X_2^\le-\nabla X_1^\le\cdot\nabla X_2^>-\nabla X_1^>\cdot\nabla X_2^\le
\end{equation}
using the computations of the previous Appendix. Thus we get
\begin{align}
Y&=\xi-\Delta X-\Wick{|\nabla X|^2}\\
&=X+\Delta\SU_\le(X_1+X_2+X_3)+2\nabla X_1^>\cdot\nabla X_3^>+2\nabla X_2^>\cdot\nabla X_3^>+|\nabla X_3^>|^2\\
&\quad-|\nabla X_1^\le|^2-2\nabla X_1^\le\cdot\nabla X_1^>-|\nabla X_2^\le|^2-2\nabla X_2^\le\cdot\nabla X_2^>\\
&\quad-\nabla X_1^\le\cdot\nabla X_2^\le-\nabla X_1^\le\cdot\nabla X_2^>-\nabla X_1^>\cdot\nabla X_2^\le
\end{align}
and a careful track of each terms give $Y\in\CC_{-\mu}^{-\frac{1}{2}-\kappa}(\IR^3)$. As in two dimensions, we have
\begin{equation}
e^X\CH e^Xv=-\nabla(e^{2X}\nabla v)+e^{2X}Yv
\end{equation}
with
\begin{equation}
\langle\CH u,u\rangle=\int_{\IR^2}|\nabla v(x)|^2e^{2X(x)}\drm x+\int_{\IR^2}|v(x)|^2Y(x)e^{2X(x)}\drm x
\end{equation}
where the second integral is interpreted as a duality bracket since $Y$ is not a function, well-defined for $v\in H_\delta^{\frac{1}{2}+2\kappa}(\IR^3)$.

\medskip

{\bf Paracontrolled calculus.} We do not give here a general introduction on paracontrolled calculus, see for example the seminal work \cite{GIP} or the lecture notes \cite{GubinelliPerkowskiNotes16} as well as references therein. In this work, we only need the definition of the paraproduct and resonant terms with their continuity results that we recall here. Given Paley-Littlewood projectors defined before, one can consider formally the product as
\begin{align}
u\cdot v&=\sum_{n,m\ge0}\Delta_nu\cdot\Delta_mv\\
&=\sum_{n<m-1}\Delta_nu\cdot\Delta_mv+\sum_{|n-m|\le 1}\Delta_nu\cdot\Delta_mv+\sum_{n>m+1}\Delta_nu\cdot\Delta_mv\\
&=:\P_uv+\PI(u,v)+\P_vu
\end{align}
with $\P_uv$ the paraproduct of $v$ by $u$ and $\PI(u,v)$ the resonant term. The following continuity result can be proved, this can be proved for example with a proof in the same lines as Lemma $2.1$ in \cite{GIP}.

\medskip

\begin{proposition}
Let $\alpha,\beta\in\IR$ and $\mu,\nu\in\IR$. We have
\begin{equation}
\|\P_uv\|_{H_{\mu+\nu}^{\alpha\wedge0+\beta}(\IR^d)}\lesssim\|u\|_{H_\mu^\alpha(\IR^d)}\|v\|_{\CC_\nu^\beta(\IR^d)}
\end{equation}
and
\begin{equation}
\|\P_uv\|_{H_{\mu+\nu}^{\alpha\wedge0+\beta}(\IR^d)}\lesssim\|u\|_{\CC_\mu^\alpha(\IR^d)}\|v\|_{H_\nu^\beta(\IR^d)}.
\end{equation}
If moreover $\alpha+\beta>0$, we get
\begin{equation}
\|\PI(u,v)\|_{H_{\mu+\nu}^{\alpha+\beta}}\lesssim\|u\|_{H_\mu^\alpha(\IR^d)}\|v\|_{\CC_\nu^\beta(\IR^d)}.
\end{equation}
\end{proposition}

\medskip

In addition, one also needs to bound the corrector
\begin{equation}
\DC(u,v,w)=\PI(\P_uv,w)-u\PI(v,w)
\end{equation}
introduced initially in \cite{GIP}. We have the following continuity result which is proved in a similar proof.

\medskip

\begin{proposition}
Let $\alpha,\beta,\gamma\in\IR$ such that $\beta+\gamma<0<\alpha+\beta+\gamma$ and $\mu_0,\mu_1,\mu_2>0$. The trilinear operator $\DC$ extends uniquely as a continuous map from $H_{\mu_0}^\alpha(\IR^d)\times\CC_{\mu_1}^\beta(\IR^d)\times\CC_{\mu_2}^\gamma(\IR^d)$ to $H_{\mu_0+\mu_1+\mu_2}^{\alpha+\beta+\gamma}(\IR^d)$.
\end{proposition}

\medskip

To obtain Strichartz inequalities for the linear equation, our proof follows a perturbative argument with respect to the Laplacian. Given the previous construction with \eqref{EqExpressionH} which makes sense for any $v\in H_\delta^{\frac{d}{2}-1+2\kappa}(\IR^d)$, we have
\begin{equation}\label{ExpH}
e^{-X}\CH e^Xv=-\Delta v-2\nabla X\cdot\nabla v+Yv
\end{equation}
if $v$ is smooth enough for the cross term to be well-defined, that is the stronger condition $v\in H_\delta^{\frac{d}{2}+2\kappa}(\IR^d)$. This gives the bound
\begin{equation}
\|(e^{-X}\CH e^X+\Delta)v\|_{H^{1-\frac{d}{2}-\kappa}(\IR^d)}\lesssim\|v\|_{H^{\frac{d}{2}+\gamma}(\IR^d)\cap L_\mu^2(\IR^d)}
\end{equation}
for any $\gamma>0$. To obtain Strichartz inequalities, we use paracontrolled calculus to improve this bound to gain regularity using $\gamma$ with a conjugated operator $\CH^\sharp$, as done in our previous work \cite{MZ} on compact surfaces, see also \cite{DebusscheMouzard24,DeVecchiJiZachhuber25}. We introduce the operator
\begin{equation}
\CH^\sharp:=(e^X\Gamma)^{-1}\CH(e^X\Gamma)
\end{equation}
for a suitable random map $\Gamma$ in order to improve the previous comparison. This corresponds to a paracontrolled ansatz based on \eqref{ExpH}, again more involved in three dimensions. For $d=2$, we consider
\begin{equation}
\Phi(v):=v-\P_{\nabla v}Z_1-\P_vZ_2
\end{equation}
with
\begin{align}
Z_1&=2(1-\Delta)^{-1}\nabla(X-X_N)\in\CC_\delta^{2-\kappa}(\IR^2),\\
Z_2&=(1-\Delta)^{-1}(Y^>-Y_N^>)\in\CC_\delta^{2-\kappa}(\IR^2)
\end{align}
with $X_N$ and $Y_N^>$ smooth approximations respectively of $X$ and $Y^>:=\SU_>Y$. In three dimensions, the higher order paracontrolled expansion
\begin{equation}
\Phi(v):=v-\P_{\nabla v}(Z_1+Z_3)-\P_v(Z_2+Z_4)
\end{equation}
is needed with
\begin{align}
Z_1&=2(1-\Delta)^{-1}\nabla(X-X_N)\in\CC_\delta^{\frac{3}{2}-\kappa}(\IR^3),\\
Z_2&=(1-\Delta)^{-1}(Y^>-Y_N^>)\in\CC_\delta^{\frac{3}{2}-\kappa}(\IR^3),\\
Z_3&=2(1-\Delta)^{-1}\big(\PI(\nabla X,Z_1)+\P_{\nabla X}Z_1\big)\in\CC_\delta^{2-\kappa}(\IR^3),\\
Z_4&=2(1-\Delta)^{-1}\big(\PI(\nabla X,Z_2)+\P_{\nabla X}Z_2\big)\in\CC_\delta^{2-\kappa}(\IR^3).
\end{align}
The map $\Phi$ is a compact perturbation of the identity on $H^\sigma(\IR^d)$ for $\sigma<2$ in both cases. While it is not a priori invertible, one can truncate the stochastic terms in order to have a small pertubation of the identity without impacting the regularity of the objects involved. We do not keep the dependance with respect to $N$ in the notation to simplify the computations however the map $\Phi=\Phi_N$ converges to the identity as $N$ goes to infinity thus this map is invertible for $N=N(X,Y_>)$ large enough. For such parameter fixed high enough, we denote as $\Gamma:H^\sigma\to H^\sigma$ its inverse as done in \cite{GUZ}, see also \cite{Mouzard} for a different choice of truncation. Since $X_N$ and $Y_{>,N}$ are smooth functions, this truncation has no impact on the regularity study needed for this work and the map $\Gamma$ is defined by an implicit relation, for example
\begin{equation}
\Gamma v^\sharp=\P_{\nabla\Gamma v^\sharp}2\nabla(1-\Delta)^{-1}(X-X_N)+\P_{\Gamma v^\sharp}(1-\Delta)^{-1}(-Y^>+Y_N^>)+v^\sharp
\end{equation}
with $v^\sharp\in L^2(\IR^2)$ in two dimensions. It allows to obtain the following proposition.

\medskip

\begin{proposition}\label{prop:fineHcomparison}
Let $\mu>0$ and $\gamma\in(0,2-\frac{d}{2})$. We have
\begin{equation}
\|(\CH^\sharp+\Delta)v^\sharp\|_{H^{\gamma}(\IR^d)}\lesssim\|v^\sharp\|_{H^{\frac{d}{2}+\gamma+\kappa}(\IR^d)\cap L_\mu^2(\IR^d)}
\end{equation}
for any $\kappa>0$.
\end{proposition}

\medskip

\begin{proof}
The operator is given by
\begin{equation}
e^{-X}\CH e^Xv=-\Delta v-2\nabla X\cdot\nabla v+Yv
\end{equation}
and we now assume that $v=\Gamma v^\sharp$ with $v^\sharp\in H^{\frac{d}{2}+\gamma}(\IR^d)$ with $\gamma>0$. We first consider $d=2$, the computations are heavier but similar in three dimensions. We have
\begin{equation}
v=\P_{\nabla v}Z_1+\P_vZ_2+v^\sharp
\end{equation}
with $(1-\Delta)Z_1=\nabla(X-X_N)$ and $(1-\Delta)Z_2=-Y^>+Y_N^>$ thus
\begin{align}
e^{-X}\CH e^X\Gamma v^\sharp&=-\Delta v^\sharp-\Delta\P_{\nabla v}Z_1-\Delta\P_vZ_2-2\nabla X\cdot\nabla v+Yv\\
&=-\Delta v^\sharp-\P_{\nabla v}\Delta Z_1-\P_v\Delta Z_2-[\Delta,\P_{\nabla v}]Z_1-[\Delta,\P_v]Z_2\\
&\quad-2\P_{\nabla X}\nabla v-2\PI(\nabla X,\nabla v)-\P_{\nabla v}2\nabla X+\P_Yv+\PI(Y,v)+\P_vY\\
&=-\Delta v^\sharp-\P_{\nabla v}Z_1-\P_v Z_2-[\Delta,\P_{\nabla v}]Z_1-[\Delta,\P_v]Z_2\\
&\quad-2\P_{\nabla X}\nabla v-2\PI(\nabla X,\nabla v)-\P_{\nabla v}2\nabla X_N+\P_Yv+\PI(Y,v)+\P_v(Y_N^>+Y^\le)
\end{align}
where the choice of $Z_1$ and $Z_2$ cancels the roughest term with $Y^\le:=\SU_\le Y$. We get
\begin{align}
(e^{-X}\CH e^X\Gamma+\Delta)v^\sharp&=-\P_{\nabla v}Z_1-\P_v Z_2-[\Delta,\P_{\nabla v}]Z_1-[\Delta,\P_v]Z_2-2\P_{\nabla X}\nabla v-2\PI(\nabla X,\nabla v)\\
&\quad-\P_{\nabla v}2\nabla X_N+\P_Yv+\PI(Y,v)+\P_v(Y_N^>+Y^\le)
\end{align}
where the roughest term are $\PI(\nabla X,\nabla v)$ and $\P_{\nabla X}\nabla v$. Since $v\in H_\delta^{1+\gamma}(\IR^2)$ with $\gamma>0$, these terms belong to $H^{\gamma-\kappa}(\IR^2)$ which is the claimed result up to applying $\Gamma^{-1}$ which is the identity up to a regularizing perturbation. In three dimensions, we have 
\begin{equation}
v=\P_{\nabla v}(Z_1+Z_3)+\P_v(Z_2+Z_4)+v^\sharp
\end{equation}
with $(1-\Delta)Z_3=2\PI(\nabla X,Z_1)+2\P_{\nabla X}Z_1$ and $(1-\Delta)Z_4=2\PI(\nabla X,Z_2)+2\P_{\nabla X}Z_2$. Using the previous computation, we get
\begin{align}
e^{-X}\CH e^X\Gamma v^\sharp&=-\Delta v^\sharp-\P_{\nabla v}(Z_1+Z_3)-\P_v (Z_2+Z_4)-[\Delta,\P_{\nabla v}](Z_1+Z_3)-[\Delta,\P_v](Z_2+Z_4)\\
&\quad+\P_{\nabla v}(1-\Delta)Z_3+\P_v(1-\Delta)Z_4-2\PI(\nabla X,\nabla v)\\
&\quad-2\P_{\nabla X}\nabla v-\P_{\nabla v}2\nabla X_N+\P_Yv+\PI(Y,v)+\P_v(Y_N^>+Y^\le)
\end{align}
where the terms of the last line are well-controlled. Using the paracontrolled expansion for $v$, we get
\begin{equation}
\nabla v=\P_{\nabla v}\nabla Z_1+\P_{\Delta v}Z_1+\P_v\nabla Z_2+\P_{\nabla v}Z_2+\nabla v^\sharp
\end{equation}
hence
\begin{equation}
\PI(\nabla X,\nabla v)=\PI(\nabla X,\P_{\nabla v}\nabla Z_1)+\PI(\nabla X,\P_v\nabla Z_2)+\PI(\nabla X,\P_{\Delta v}Z_1+\P_{\nabla v}Z_2+\nabla v^\sharp)
\end{equation}
where the last term is well-defined for $\gamma>\frac{1}{2}$ since $\nabla X\in\CC_\delta^{-\frac{1}{2}-\kappa}(\IR^3)$. For the first two terms, we use the corrector from paracontrolled calculus to get
\begin{equation}
\PI(\nabla X,\P_{\nabla v}\nabla Z_i)=\nabla v\PI(\nabla X,\nabla Z_i)+\DC(\nabla v,\nabla X,\nabla Z_i)
\end{equation}
for $i\in\{1,2\}$. The corrector is well-defined since $\nabla v\in H^\gamma(\IR^3),\nabla X\in\CC_\delta^{-\frac{1}{2}-\kappa}(\IR^3),\nabla Z_i\in\CC_\delta^{\frac{1}{2}-\kappa}(\IR^3)$ thus the sum of the exponents is positive for $\gamma>2\kappa$. The product $\PI(\nabla X,\nabla Z_i)\in\CC_\delta^{-\kappa}(\IR^3)$ are defined as the limit of the regularization without renormalization content, as done in the previous section. We get
\begin{align}
e^{-X}\CH e^X\Gamma v^\sharp&=-\Delta v^\sharp-\P_{\nabla v}(Z_1+Z_3)-\P_v (Z_2+Z_4)-[\Delta,\P_{\nabla v}](Z_1+Z_3)-[\Delta,\P_v](Z_2+Z_4)\\
&\quad+\P_{\nabla v}(1-\Delta)Z_3+\nabla v\PI(\nabla X,\nabla Z_1)+\DC(\nabla v,\nabla X,\nabla Z_1)\\
&\quad+\P_{\nabla v}(1-\Delta)Z_4+\nabla v\PI(\nabla X,\nabla Z_2)+\DC(\nabla v,\nabla X,\nabla Z_2)\\
&\quad+\PI(\nabla X,\P_{\Delta v}Z_1+\P_{\nabla v}Z_2+\nabla v^\sharp)\\
&\quad-2\P_{\nabla X}\nabla v-\P_{\nabla v}2\nabla X_N+\P_Yv+\PI(Y,v)+\P_v(Y_N^>+Y^\le)\\
&=-\Delta v^\sharp-\P_{\nabla v}(Z_1+Z_3)-\P_v (Z_2+Z_4)-[\Delta,\P_{\nabla v}](Z_1+Z_3)-[\Delta,\P_v](Z_2+Z_4)\\
&\quad+\P_{\PI(\nabla X,\nabla Z_1)}\nabla v+\PI(\nabla v,\PI(\nabla X,\nabla Z_1))+\DC(\nabla v,\nabla X,\nabla Z_1)\\
&\quad+\P_{\PI(\nabla X,\nabla Z_2)}\nabla v+\PI(\nabla v,\PI(\nabla X,\nabla Z_2))+\DC(\nabla v,\nabla X,\nabla Z_2)\\
&\quad+\PI(\nabla X,\P_{\Delta v}Z_1+\P_{\nabla v}Z_2+\nabla v^\sharp)\\
&\quad-2\P_{\nabla X}\nabla v-\P_{\nabla v}2\nabla X_N+\P_Yv+\PI(Y,v)+\P_v(Y_N^>+Y^\le)
\end{align}
using again the paraproduct decomposition for products with $\PI(\nabla X,\nabla Z_i)$ and the cancellation with $Z_3$ and $Z_4$. Tracking the precise regularity of each terms yield the result up to applying again $\Gamma^{-1}$. For example, the roughest term is
\begin{equation}
\|\PI(\nabla X,\nabla v^\sharp)\|_{H^{\gamma-\frac{1}{2}-\kappa}(\IR^3)}\lesssim\|\nabla X\|_{\CC_\delta^{-\frac{1}{2}-\kappa}(\IR^3)}\|v^\sharp\|_{H^{1+\gamma}(\IR^3)}
\end{equation}
where $\gamma>\frac{1}{2}$ which is a loss of $\frac{3}{2}$ derivatives and completes the proof.
\end{proof}

While $\Gamma e^{X}H^2(\IR^d)\cap L_\mu^2(\IR^d)$ is not the domain of the Anderson Hamiltonian in full space due to the growth condition at infinity, one can show that this is a core however we do not pursue any spectral theory in this work. The following Corollary states that given $u\in L_\mu^2(\IR^d)$, it is equivalent to bound $v\in\Gamma H^2(\IR^d)$ or $e^X\CH e^Xv\in L^2(\IR^d)$ with $v=\Gamma v^\sharp$.

\medskip

\begin{corollary}\label{cor:EquivalenceDomainSpaces}
For any $\mu,\kappa>0$, one has
\begin{equation}
\|e^X\CH e^Xv+\Delta \Gamma^{-1}v\|_{L^2(\IR^d)}\lesssim\|v\|_{\Gamma H^{\frac{d}{2}+\kappa}\cap L_\mu^2}.
\end{equation}
\end{corollary}

% \bibliographystyle{siam}
% \bibliography{biblio.bib}

\begin{thebibliography}{10}

\bibitem{BCD}
{\sc H.~Bahouri, J.-Y. Chemin, and R.~Danchin}, {\em Fourier analysis and
  nonlinear partial differential equations}, vol.~343 of Grundlehren Math.
  Wiss., Berlin: Heidelberg, 2011.

\bibitem{BonjioanniTorrea06}
{\sc B.~Bongioanni and J.~L. Torrea}, {\em Sobolev spaces associated to the
  harmonic oscillator}, Proc. Indian Acad. Sci., Math. Sci., 116 (2006),
  pp.~337--360.

\bibitem{BGT}
{\sc N.~{Burq}, P.~{G\'erard}, and N.~{Tzvetkov}}, {\em {Strichartz
  inequalities and the nonlinear Schr\"odinger equation on compact manifolds}},
  {Am. J. Math.}, 126 (2004), pp.~569--605.

\bibitem{ChauleurMouzard23}
{\sc Q.~Chauleur and A.~Mouzard}, {\em The logarithmic {Schr{\"o}dinger}
  equation with spatial white noise on the full space}, J. Evol. Equ., 25
  (2025), p.~28.
\newblock Id/No 1.

\bibitem{DRTV24}
{\sc A.~Debussche, R.~Liu, N.~Tzvetkov, and N.~Visciglia}, {\em Global
  well-posedness of the 2d nonlinear {Schr{\"o}dinger} equation with
  multiplicative spatial white noise on the full space}, Probab. Theory Relat.
  Fields, 189 (2024), pp.~1161--1218.

\bibitem{DebusscheMartin19}
{\sc A.~Debussche and J.~Martin}, {\em Solution to the stochastic
  {Schr{\"o}dinger} equation on the full space}, Nonlinearity, 32 (2019),
  pp.~1147--1174.

\bibitem{DebusscheMouzard24}
{\sc A.~Debussche and A.~Mouzard}, {\em Periodic nonlinear {Schr{\"o}dinger}
  equation with distributional potential and invariant measures},  (2024).

\bibitem{DW}
{\sc A.~Debussche and H.~Weber}, {\em The {S}chr\"{o}dinger equation with
  spatial white noise potential}, Electron. J. Probab., 23 (2018), pp.~Paper
  No. 28, 16.

\bibitem{edmunds1996}
{\sc D.~E. Edmunds and H.~Triebel}, {\em Function spaces, entropy numbers,
  differential operators}, vol.~120 of Cambridge Tracts in Mathematics,
  Cambridge University Press, Cambridge, 1996.

\bibitem{EulryMouzard25}
{\sc H.~Eulry and A.~Mouzard}, {\em Ergodicity of the {Anderson} ${{\Phi_2}}^4$
  model},  (2025).

\bibitem{GinibreVelo85}
{\sc J.~Ginibre and G.~Velo}, {\em The global {Cauchy} problem for the
  nonlinear {Schr{\"o}dinger} equation revisited}, Ann. Inst. Henri
  Poincar{\'e}, Anal. Non Lin{\'e}aire, 2 (1985), pp.~309--327.

\bibitem{GubinelliHofmanova19}
{\sc M.~Gubinelli and M.~Hofmanov{\'a}}, {\em Global solutions to elliptic and
  parabolic {{\({\Phi^4}\)}} models in {Euclidean} space}, Commun. Math. Phys.,
  368 (2019), pp.~1201--1266.

\bibitem{GIP}
{\sc M.~Gubinelli, P.~Imkeller, and N.~Perkowski}, {\em Paracontrolled
  distributions and singular {PDE}s}, Forum Math. Pi, 3 (2015), pp.~e6, 75.

\bibitem{GubinelliPerkowski2017}
{\sc M.~Gubinelli and N.~Perkowski}, {\em {KPZ} reloaded}, Commun. Math. Phys.,
  349 (2017), pp.~165--269.

\bibitem{GubinelliPerkowskiNotes16}
\leavevmode\vrule height 2pt depth -1.6pt width 23pt, {\em An introduction to
  singular {SPDEs}},  (2018), pp.~69--99.

\bibitem{GUZ}
{\sc M.~Gubinelli, B.~Ugurcan, and I.~Zachhuber}, {\em Semilinear evolution
  equations for the {A}nderson {H}amiltonian in two and three dimensions},
  Stoch. Partial Differ. Equ. Anal. Comput., 8 (2020), pp.~82--149.

\bibitem{Hai14}
{\sc M.~Hairer}, {\em A theory of regularity structures}, Invent. Math., 198
  (2014), pp.~269--504.

\bibitem{HairerLabbe15}
{\sc M.~Hairer and C.~Labb{\'e}}, {\em A simple construction of the continuum
  parabolic {Anderson} model on {{\(\mathbf{R}^2\)}}}, Electron. Commun.
  Probab., 20 (2015), p.~11.
\newblock Id/No 43.

\bibitem{HsuLabbe25}
{\sc Y.-S. Hsu and C.~Labb{\'e}}, {\em Construction and spectrum of the
  {Anderson} {Hamiltonian} with white noise potential on $\mathbb{R}^2$ and
  $\mathbb{R}^3$},  (2025).

\bibitem{KeelTao1998}
{\sc M.~Keel and T.~Tao}, {\em Endpoint {S}trichartz estimates}, Amer. J.
  Math., 120 (1998), pp.~955--980.

\bibitem{LiuTzvetkov2026}
{\sc R.~Liu and N.~Tzvetkov}, {\em Large torus limit of global dynamics of the
  two-dimensional dispersive {Anderson} model},  (2026).

\bibitem{Mackowiak25}
{\sc S.~Mackowiak}, {\em Local wellposedness of the 2d
  anderson-gross-pitaevskii equation},  (2025).

\bibitem{Mackowiak2025}
{\sc S.~Mackowiak}, {\em Wellposedness of the cubic {Gross}-{Pitaevskii}
  equation with spatial white noise on {{\(\mathbb{R}^2\)}}}, Nonlinearity, 38
  (2025), p.~40.
\newblock Id/No 035026.

\bibitem{MatsudaZuijlen22}
{\sc T.~Matsuda and W.~van Zuijlen}, {\em Anderson hamiltonians with singular
  potentials}, 2022.

\bibitem{Mouzard}
{\sc A.~Mouzard}, {\em Weyl law for the {Anderson} {Hamiltonian} on a
  two-dimensional manifold}, Ann. Inst. Henri Poincar{\'e}, Probab. Stat., 58
  (2022), pp.~1385--1425.

\bibitem{MouzardOuhabaz23}
{\sc A.~Mouzard and E.~M. Ouhabaz}, {\em A simple construction of the
  {Anderson} operator via its quadratic form in dimensions two and three}, C.
  R., Math., Acad. Sci. Paris, 363 (2025), pp.~183--197.

\bibitem{MZ}
{\sc A.~Mouzard and I.~Zachhuber}, {\em Strichartz inequalities with white
  noise potential on compact surfaces}, Anal. PDE, 17 (2024), pp.~421--454.

\bibitem{sickel2014}
{\sc W.~Sickel, L.~Skrzypczak, and J.~Vyb\'{\i}ral}, {\em Complex interpolation
  of weighted {B}esov and {L}izorkin-{T}riebel spaces}, Acta Math. Sin. (Engl.
  Ser.), 30 (2014), pp.~1297--1323.

\bibitem{triebel1983}
{\sc H.~Triebel}, {\em Theory of function spaces}, vol.~78 of Monographs in
  Mathematics, Birkh\"{a}user Verlag, Basel, 1983.

\bibitem{TzvetkovVisciglia23}
{\sc N.~Tzvetkov and N.~Visciglia}, {\em Global dynamics of the {{\(2d\)}}
  {NLS} with white noise potential and generic polynomial nonlinearity},
  Commun. Math. Phys., 401 (2023), pp.~3109--3121.

\bibitem{TzvetkovVisciglia23bis}
\leavevmode\vrule height 2pt depth -1.6pt width 23pt, {\em Two dimensional
  nonlinear {Schr{\"o}dinger} equation with spatial white noise potential and
  fourth order nonlinearity}, Stoch. Partial Differ. Equ., Anal. Comput., 11
  (2023), pp.~948--987.

\bibitem{Ueki25}
{\sc N.~Ueki}, {\em A definition of self-adjoint operators derived from the
  {Schr{\"o}dinger} operator with the white noise potential on the plane},
  Stochastic Processes Appl., 186 (2025), p.~30.
\newblock Id/No 104642.

\bibitem{DeVecchiJiZachhuber25}
{\sc F.~C.~D. Vecchi, X.~Ji, and I.~Zachhuber}, {\em Stochastic hartree nls in
  3d coming from a many-body quantum system with white noise potential},
  (2025).

\bibitem{Yosida74}
{\sc K.~Yosida}, {\em Functional analysis. 4th ed}, vol.~123 of Grundlehren
  Math. Wiss., Springer, Cham, 1974.

\end{thebibliography}

\vspace{2cm}

\noindent \textcolor{gray}{$\bullet$} A. Mouzard -- Modal’X, Université Paris Nanterre, 92000 Nanterre, France.\\
{\it E-mail}: antoine.mouzard@math.cnrs.fr

\smallskip

\noindent \textcolor{gray}{$\bullet$} I. Zachhuber -- Institut für Mathematik, FU Berlin.\\
{\it E-mail}: imzach@zedat.fu-berlin.de
 
\end{document}